\documentclass[english]{ourlematema}
\usepackage{amsmath, amsthm, amssymb, hyperref, color}
\usepackage{graphicx}
\usepackage{nicematrix}
\usepackage{caption}
\usepackage{verbatim}
\usepackage{chemfig, chemnum}
\usepackage{tikz}
\usepackage{multicol}
\usepackage{algorithm2e}
\usepackage{caption}
\usepackage{subcaption}
\usepackage{float}
\usepackage{mathtools}
\usepackage{braket}
\usepackage{cleveref}
\usepackage{anyfontsize}
\usepackage{todonotes}
\usepackage{tikz-cd}
\usepackage{diagbox}
\usepackage{enumitem}
\setenumerate{label=\textup{(\roman*)}}

\newtheorem{theorem}{Theorem}
\numberwithin{theorem}{section}
\newtheorem{proposition}[theorem]{Proposition}
\newtheorem{lemma}[theorem]{Lemma}
\newtheorem{corollary}[theorem]{Corollary}

\newtheorem{remark}[theorem]{Remark}
\newtheorem{examplex}[theorem]{Example}
\newtheorem{conjecture}[theorem]{Conjecture}
\newenvironment{example}
{\pushQED{\qed}\examplex}
{\popQED\endexamplex}

\newcommand{\ZZ}{\mathbb{Z}}

\newcommand{\RR}{\mathbb{R}}
\newcommand{\CC}{\mathbb{C}}
\newcommand{\PP}{\mathbb{P}}

\newcommand\blankfootnote[1]{%
  \let\thefootnote\relax\footnotetext{#1}%
  \let\thefootnote\svthefootnote%
}

 \title{Proudfoot-Speyer degenerations of ~\\ scattering equations}

  \author{Barbara Betti}
  \address{%
  MPI for Mathematics in the Sciences, Leipzig \\
\email{barbara.betti@mis.mpg.de}
}

  \author{Viktoriia Borovik}
  \address{%
  University of Osnabr\"uck \\
\email{vborovik@uni-osnabrueck.de}
}

\author{Simon Telen}
\address{%
  MPI for Mathematics in the Sciences, Leipzig \\
\email{simon.telen@mis.mpg.de}
}

\date{28/08/2024}

\begin{document}
\maketitle

\begin{abstract}
\noindent We study scattering equations of hyperplane arrangements from the perspective of combinatorial commutative algebra and numerical algebraic geometry. We formulate the problem as linear equations on a reciprocal linear space and develop a degeneration-based homotopy algorithm for solving them. We investigate the Hilbert regularity of the corresponding homogeneous ideal and apply our methods to CHY scattering equations. %amplitudes.
\end{abstract}

\section{Introduction}
\blankfootnote{\textit{Keywords.} hyperplane arrangements, scattering equations, maximum likelihood estimation, homotopy continuation, Hilbert regularity}\blankfootnote{\textit{2020 Mathematics Subject Classification.} 52C35, 65H14, 13F55}Consider $n+1$ hyperplanes in $\mathbb{C}^d$, defined by $\ell_0(x) = 0, \, \ldots, \, \ell_n(x) = 0$. Here, $\ell_i \in \mathbb{C}[x_1, \ldots, x_d]$ are affine-linear functions in $d$ variables. 
The following logarithmic potential function serves as the scattering potential in CHY theory~\cite{cachazo2014scattering}:
\begin{equation}  \label{eq:loglikelihood_intro}
{\cal L}_u (x) \, = \, \log \ell_0^{u_0} \ell_1^{u_1} \cdots \ell_n^{u_n} \, = \, u_0 \log \ell_0 + u_1 \log \ell_1 + \cdots + u_n \log \ell_n. \end{equation}
This function depends on parameters $u = (u_0, \ldots, u_n)$ which take complex values. Motivated by the physics application, see for instance \cite{cachazo2014scattering,lam2024moduli,sturmfels2021likelihood}, we are interested in solving its critical point equations for generic $u$: 
\begin{equation} \label{eq:scatteringeqs}  \frac{\partial{\cal L}_u}{\partial x_1} \, = \, \cdots \, = \,  \frac{\partial{\cal L}_u}{\partial x_d}\, = \, 0.  \end{equation}
Notice that these equations are invariant under scaling the linear forms $\ell_i$ by non-zero constants; they only depend on the arrangement ${\cal A} = V(\ell_0 \cdots \ell_n) \subset \mathbb{C}^d$ and are given by well-defined rational functions on its complement $X = \mathbb{C}^d \setminus {\cal A}$. We refer to~\eqref{eq:scatteringeqs} as the scattering equations of the hyperplane arrangement ${\cal A}$. 

We collect the coefficients of $\ell_0, \ldots, \ell_n$ in a matrix $L \in \mathbb{C}^{(d+1) \times (n+1)} $ such~that 
\[ L^\top \, = \, \begin{pmatrix}
    -b & A
\end{pmatrix} \quad \text{and} \quad (\ell_0(x), \ldots, \ell_n(x)) \, = \, (1, x_1, \ldots, x_d) \cdot L.\]
Here $A$ has size $(n+1) \times d$, and $\ell_i(x)$ is the $i$-th entry of $A \, x - b$ (counting starts at zero). With this notation, we can rewrite the scattering equations \eqref{eq:scatteringeqs} as 
\begin{equation}  \label{eq:embeddedeq} A^\top {\rm diag}(u)  \, y \, = \, 0 \quad \text{and} \quad y \in {\rm im} \, \phi, \end{equation}
where $\phi: X \rightarrow \mathbb{P}^n$ is the morphism $x \mapsto (\ell_0^{-1}(x): \cdots : \ell_n^{-1}(x))$ and ${\rm diag}(u)$ is an $(n+1) \times (n+1)$ diagonal matrix with diagonal entries $u_0, \ldots, u_n$. From an algebro-geometric perspective, it is natural to relax the condition $y \in {\rm im} \, \phi$ by replacing the image of $\phi$ with its closure in the projective space $\mathbb{P}^n$. This is the reciprocal linear space associated to the row span of $L$, denoted by $ {\cal R}_L$. 

In the terminology of Proudfoot and Speyer \cite{Proudfoot2006}, reciprocal linear spaces are spectra of broken circuit rings. Their geometric properties are encoded by the matroid $M(L)$ represented by $L$. This includes dimension, degree, singular locus and a nice stratification of $ {\cal R}_L$ \cite{Proudfoot2006,sanyal2013entropic}. Reciprocal linear spaces appear naturally in regularized linear programming \cite{de2012central}. A key geometric feature for our purposes is the fact that $  {\cal R}_L$ admits a Gr\"obner degeneration to a reduced union of coordinate subspaces. On the algebraic side, a universal Gr\"obner basis for the vanishing ideal $I({\cal R}_L)$ degenerates to a set of generators for a square-free monomial ideal $J$. We call this a Proudfoot-Speyer degeneration of $  {\cal R}_L$. Our paper turns such degenerations into practice. We develop a homotopy algorithm for solving \eqref{eq:embeddedeq} which starts by solving linear equations on~$V(J)$. This requires only combinatorics and linear algebra. Next, we lift the solutions in~$V(J)$ through the degeneration to the solutions of \eqref{eq:embeddedeq} in $  {\cal R}_L$. Here is an example with $d = 2, n = 3$. 

\begin{example} \label{ex:intro}
    The function ${\cal L}_u(x) = x_1^{u_0}x_2^{u_1}(2-x_1-2x_2)^{u_2}(2-2x_1-x_2)^{u_3}$ uses 
    \begin{equation} \label{eq:Lintro} L \, = \, \begin{pmatrix}
        0 & 0 & 2 & 2 \\ 1 & 0 & -1 & -2 \\ 0 & 1 & -2 & -1
    \end{pmatrix}.\end{equation}
    The scattering equations are two rational function equations in two unknowns:
    \[ \frac{u_0}{x_1} - \frac{u_2}{2-x_1-2x_2} - \frac{2\, u_3}{2-2x_1-x_2} \, = \, \frac{u_1}{x_2} - \frac{2\, u_2}{2-x_1-2x_2} - \frac{ u_3}{2-2x_1-x_2} \, = \, 0.\]
     In coordinates $y_i = \ell_i^{-1}$, the system of equations \eqref{eq:embeddedeq}, replacing ${\rm im} \, \phi$ with ${\cal R}_L$, is
    \begin{equation}  \label{eq:scatteringexample} \begin{matrix} u_0 y_0 -u_2y_2 -2 u_3y_3 \, = \, u_1y_1 -2u_2y_2 -u_3y_3 \, = \, 0, \\ 
    y_1y_2y_3 - y_0 y_2y_3 - y_0y_1y_3 + y_0y_1y_2 \, = \, 0. \end{matrix} \end{equation}
    The last equation defines the cubic surface $  {\cal R}_L \subset \mathbb{P}^3$. After fixing generic values for $u$, the first two equations define a line $\mathbb{L}_u$ in $\mathbb{P}^3$. This line hits ${\rm im} \, \phi \subset   {\cal R}_L$ in three points. Their pre-images under $\phi$ are the three solutions to the scattering equations of ${\cal A}$ (see Figure \ref{fig:4lines}). A Proudfoot-Speyer degeneration of ${\cal R}_L$ is
    \[  y_1y_2y_3 - t^{1} \, y_0 y_2y_3 - t^{2} \, y_0y_1y_3 + t^3 \, y_0y_1y_2 \, = \, 0. \]
    For $t = 1$, this is the equation for ${\cal R}_L$. For $t = 0$, this defines the union of three coordinate planes $V(J)$, where $J = \langle y_1y_2y_3 \rangle$. Our line $\mathbb{L}_u$ hits each of these planes in a single point. These points are easily computed by solving linear equations. As $t$ varies from $0$ to $1$, the points in $\mathbb{L}_u \cap V(J)$ move to the solutions of \eqref{eq:scatteringexample}. A homotopy algorithm tracks the points numerically as~$t \rightarrow 1$.
\end{example}

The data from Example \ref{ex:intro} are particularly nice. In general, it might be necessary to vary the linear space $\mathbb{L}_u \, = \, \{ y \, : \, A^\top \, {\rm diag}(u) \, y = 0 \}$ throughout the homotopy, and $\mathbb{L}_u \cap {\cal R}_L$ may contain more points than the solution set of \eqref{eq:embeddedeq}. That is, for certain choices of $L$, the scattering equations always have solutions on the boundary ${\cal R}_L \setminus {\rm im} \, \phi$, see Theorem \ref{thm:flats}. Generically, this does not happen. 

\begin{theorem} \label{thm:optimal_intro}
    For generic $L \in \mathbb{C}^{(d+1) \times (n+1)}$ and generic $u \in \mathbb{C}^{n+1}$, the Proudfoot-Speyer homotopy described in Section \ref{sec:4} is optimal, meaning that the number of solution paths tracked equals the number of solutions to \eqref{eq:scatteringeqs} in $X$.
\end{theorem}

In algebraic statistics, the scattering equations \eqref{eq:scatteringeqs} appear in maximum likelihood estimation for discrete linear models. In that context, one assumes that $\ell_0(x) + \cdots + \ell_n(x) = 1$. The model is the intersection of the $n$-dimensional probability simplex with the image of the parametrization $x \mapsto (\ell_0(x), \ldots, \ell_n(x))$. The function ${\cal L}_u(x)$ is the log-likelihood function corresponding to an experiment in which state $i \in \{0,\ldots, n\}$ was observed $u_i$ times. Computing the maximizer of ${\cal L}_u(x)$ is a standard way of infering which distribution in the model best explains the data. We elaborate on the statistics application in Section \ref{sec:2}.

\begin{comment}
We note that our algorithm based on Proudfoot-Speyer degenerations can be used to solve any system of rational function equations of the form 
\[ \sum_{j = 0}^n \frac{c_{ij}}{\ell_j(x)} \, = \, 0, \quad i = 1, \ldots d, \quad c_{ij} \in \mathbb{C}. \]
In particular, the coefficients $c_{ij}$ need not depend on $L$. 
\end{comment}

In algebraic terms, viewing the scattering equations as linear equations on~${\cal R}_L$ means that we interpret the linear forms $A^\top \, {\rm diag}(u) \, y$ as elements of the homogeneous coordinate ring of ${\cal R}_L$. They generate an ideal denoted by \[ I_u = \big \langle {\textstyle\sum}_{j = 0}^n a_{ij}u_j \, y_j \, \mid  i = 1, \ldots, d \big \rangle \, \subset \, \mathbb{C}[{\cal R}_L] = \mathbb{C}[y_0, \ldots, y_n]/I({\cal R}_L).  \]
The complexity of computing the solutions in ${\cal R}_L$ algebraically, e.g., using Gr\"obner bases, is governed by the regularity of $I_u \subset \mathbb{C}[{\cal R}_L]$. This is classical for homogeneous equations on projective space \cite{bayer1987criterion}. %, and recent advances in the multigraded case include \cite{bender2022toric,bender2024multigraded}. 
For equations on arithmetically Cohen-Macaulay projective varieties, such as ${\cal R}_L$ \cite[Proposition 7]{Proudfoot2006}, see \cite[Theorem 5.4]{betti2025solving}. The following result bounds the Hilbert regularity.
\begin{theorem} \label{thm:hilbert_intro}
    Let $L \in \mathbb{C}^{(d+1) \times (n+1)}$ be of rank $d+1$. Let $\mathbb{L}\subset \mathbb{P}^n$ be a linear subspace of dimension $n-d$ which intersects ${\cal R}_L$ in $\deg {\cal R}_L$ many points, counting multiplicities, and let $I(\mathbb{L})$ be its ideal in $\mathbb{C}[{\cal R}_L]$. The Hilbert function 
    ${\rm HF}_{\mathbb{L} \cap {\cal R}_L}(q) = \dim_{\mathbb{C}} \mathbb{C}[{\cal R}_L]_q/I({\mathbb{L}})_q$ equals $\deg {\cal R}_L$ for all $q \geq d$. 
\end{theorem}

The outline is as follows. Section \ref{sec:2} recalls the basics on scattering equations and reciprocal linear spaces. Section \ref{sec:3} makes the genericity condition in Theorem \ref{thm:optimal_intro} precise by characterizing when the number of solutions to \eqref{eq:scatteringeqs} agrees with the degree of ${\cal R}_L$. Section \ref{sec:4} describes our homotopy algorithm and its implementation. Our Julia code is available at the MathRepo page \cite{mathrepo}. Section~\ref{sec:5} is on the relevant case for physics, where $X \simeq \mathcal{M}_{0,m}$ is the moduli space of $m$-pointed genus $0$ curves. Theorem \ref{thm:mainsec5} explains how the points in $\mathbb{L}_u \cap {\cal R}_L$ correspond to the scattering solutions for certain subsets of $m' = 4, 5, \ldots, m$ particles, assuming a conjecture on their multiplicities. Finally, Section \ref{sec:6} features a proof of Theorem \ref{thm:hilbert_intro} and Macaulay matrix constructions.

\section{Scattering equations, reciprocal linear spaces and maximum likelihood} \label{sec:2}

%\textcolor{blue}{scattering equations: \# solutions is number of bounded chambers, Varchenko, $\mathcal{M}_{0,m}$ example. Mention statistics and ML degree. Reciprocal linear spaces: Gr\"obner basis. Repeat equivalence of \eqref{eq:scatteringeqs} and \eqref{eq:embeddedeq}, end with an example where number of solutions $\neq$ degree. }

This section collects preliminary facts about the equations \eqref{eq:scatteringeqs} and about reciprocal linear spaces. We start with the number of solutions. That number is constant for almost all values of $u$, and depends only on the topology of $X = \mathbb{C}^d \setminus {\cal A}$. We assume that the hyperplane arrangement ${\cal A}$ is \emph{essential}, meaning that there is a subset of its hyperplanes which intersect in a single point. Let $\chi(X)$ be the topological Euler characteristic of $X$. The following is \cite[Theorem~1.1]{orlik1995number}.

\begin{theorem} \label{thm:orlik-terao}
    For essential ${\cal A}$ and generic values of $u \in \mathbb{C}^{n+1}$, the scattering equations \eqref{eq:scatteringeqs} have only isolated solutions. There are $(-1)^d \cdot \chi(X)$ solutions in total. Moreover, all these solutions are non-degenerate critical points of ${\cal L}_u$, meaning that the Hessian determinant $\det \left( \frac{\partial^2 {\cal L}_u}{\partial x_i \partial x_j}  \right)_{ij}$ is non-zero at each solution.
\end{theorem}

In fact, the above theorem has a much more general version \cite[Theorem 1]{Huh2013}. One can replace $\ell_i(x)$ by any polynomials so that $X = \mathbb{C}^d \setminus V(\ell_0(x) \ell_1(x) \cdots \ell_n(x))$ is a smooth very affine variety, and the statement still holds. Our focus remains on affine-linear functions $\ell_i(x)$. Theorem \ref{thm:orlik-terao} was conjectured by Varchenko, who proved the following for real arrangements, see \cite[Theorem 1.2.1]{varchenko1995critical}.

\begin{theorem} \label{thm:varchenko}
    If the matrix $L \in \mathbb{R}^{(d+1) \times (n+1)}$ is real, ${\cal A}$ is essential and $u \in \mathbb{R}^{n+1}_+$ is a tuple of positive numbers, then all solutions of the scattering equations~\eqref{eq:scatteringeqs} are real. Moreover, there is precisely one solution contained in each of the bounded chambers of the hyperplane arrangement complement $\mathbb{R}^d \setminus {\cal A}$.
\end{theorem}
\begin{comment}
        All solutions of the scattering equations~\eqref{eq:scatteringeqs} are real, and contained in the union of all bounded chambers of $\mathbb{R}^d \setminus {\cal A}$. There is precisely one solution contained in each of these bounded chambers.
\end{comment}
In particular, in the case of real arrangements, the number of bounded chambers of $\mathbb{R}^d \setminus {\cal A}$ counts the signed Euler characteristic $(-1)^d \cdot \chi(X)$ of $X = \mathbb{C}^d \setminus {\cal A}$. 

In our paper, we adopt the geometric point of view that the map  $$\phi \, : \,  x \longmapsto (\ell_0^{-1}(x): \cdots: \ell_n^{-1}(x))$$ sends the solutions of the scattering equations \eqref{eq:scatteringeqs} bijectively to the points in $\mathbb{L}_u \cap {\rm im} \, \phi$. Here $\mathbb{L}_u \subset \mathbb{P}^n$ is the linear space defined by $A^\top \, {\rm diag}(u) \, y = 0$. %Let
%\[\overline{{\rm im} \, \phi} =   \mathcal{R}_L \text{ and } \CC[\mathcal{R}_L] = \CC[s \cdot l_0^{-1}(x),\ldots, s \cdot l_n^{-1}(x)] \subset \CC(s,x_1,\ldots,x_d).\]
The closure of the image of $\phi$ in $\mathbb{P}^n$ is denoted by ${\cal R}_L$. If $L$ has rank $d+1$, then ${\cal R}_L$ is an irreducible $d$-dimensional variety called a \emph{reciprocal linear space}. 
Trivially, $\mathbb{L}_u \cap {\rm im} \, \phi$ is contained in the intersection of closed subvarieties~$\mathbb{L}_u \cap   {\cal R}_L$. %The vanishing ideal of $  {\cal R}_L$ we denote simply by $I({\cal R}_L)$.

\begin{example} \label{ex:introcontd}
    The real points of the arrangement ${\cal A}$ from Example \ref{ex:intro} are shown in Figure \ref{fig:4lines} (left). The complement $\mathbb{R}^2 \setminus {\cal A}$ has eleven
    %ten \textcolor{red}{(11?)} 
    connected components, three of which are bounded. By Theorem \ref{thm:varchenko}, for positive values of~$u \in \mathbb{R}^4_+$, the scattering equations \eqref{eq:scatteringeqs} have three real solutions. There is one solution in each of the triangles in Figure \ref{fig:4lines} (left), and one in the quadrilateral. The signed Euler characteristic of $X = \mathbb{C}^2 \setminus {\cal A}$ is three by Theorem \ref{thm:orlik-terao}, and we have seen in Example \ref{ex:intro} that this equals the degree of the reciprocal linear space ${\cal R}_L$ associated to $L$ from \eqref{eq:Lintro}. The cubic surface $\mathcal{R}_L$ is plotted in the right part of Figure \ref{fig:4lines}, together with the line $\mathbb{L}_u$ for $u = (1,1,1,1)$. The map $\phi$ sends the three solutions of scattering equations \eqref{eq:scatteringeqs} to the three points in $\mathbb{L}_u \cap {\cal R}_L$. 
    %The three ``sheets'' of the surface ${\cal R}_L$ are the images of the two triangles and the quadrilateral under $\phi$.
\end{example}
    \begin{figure}
        \centering
        \includegraphics[width = 0.4 \linewidth]{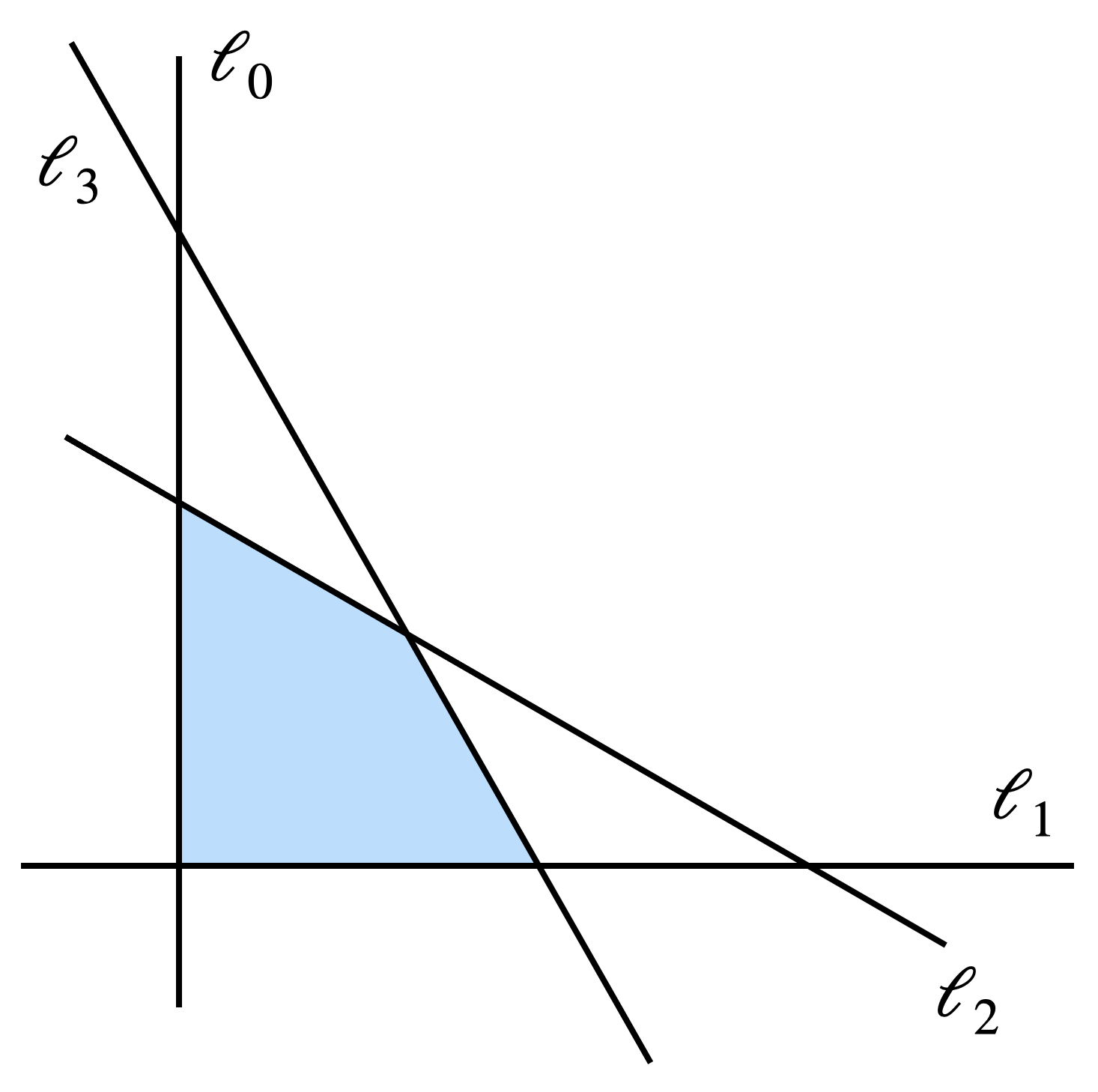}
        \quad 
        \includegraphics[width = 0.5\linewidth]{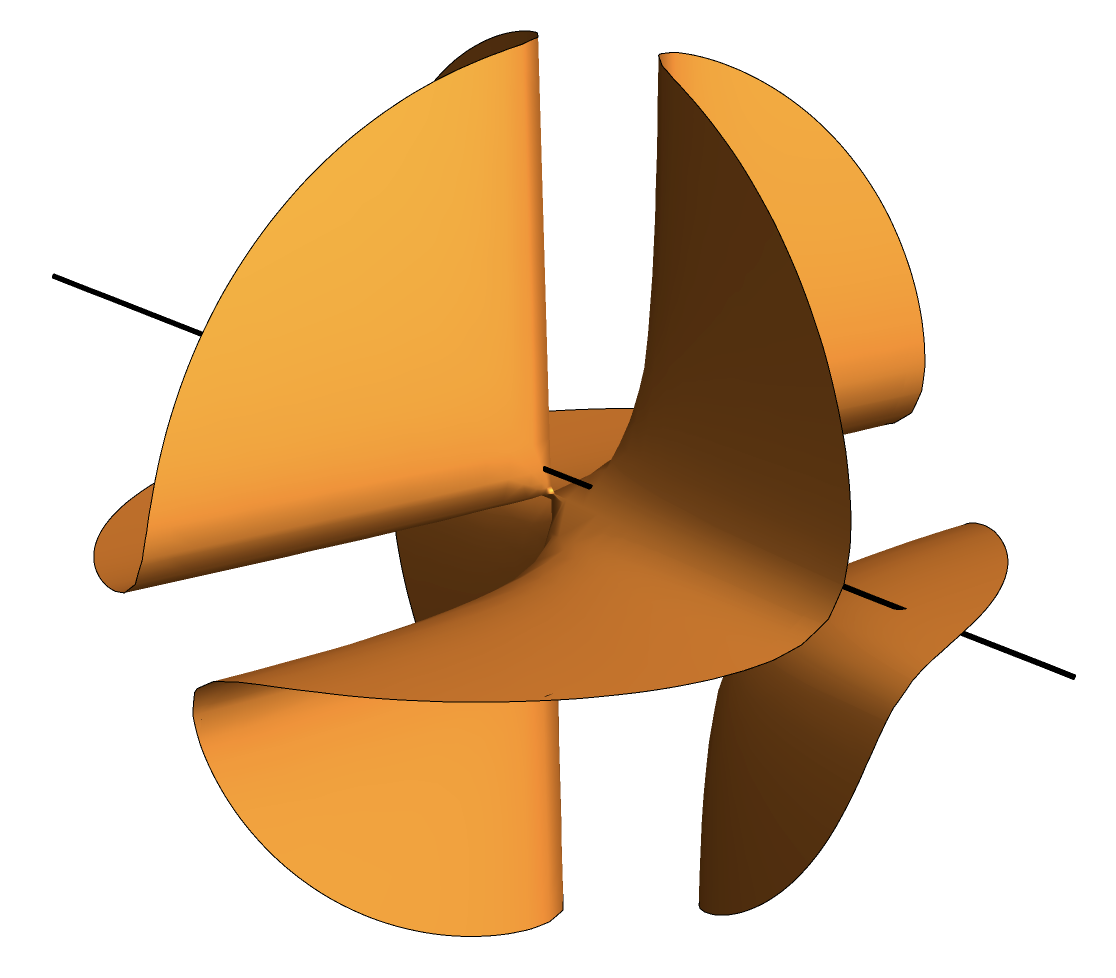}
        \caption{An arrangement of four lines in $\mathbb{R}^2$ whose scattering equations have three solutions and whose reciprocal linear space is a cubic surface in $\mathbb{P}^3$.}
        \label{fig:4lines}
    \end{figure}
The matrix $L$ defines a matroid~$M(L)$ on the ground set $\{0,\ldots,n\}$ whose dependent sets index linearly dependent columns of $L$. By the results of \cite{Proudfoot2006}, the degree of ${\cal R}_L$ only depends on the matroid $M(L)$. If $M(L)$ is uniform, like in Example \ref{ex:introcontd}, then $\deg {\cal R}_L$ is given by the binomial coefficient $\binom{n}{d}$.

The equality $\deg {\cal R}_L = (-1)^d \cdot \chi(X)$ holds in Example \ref{ex:introcontd}, but might fail in general. Theorem \ref{thm:orlik-terao} and injectivity of $\phi: X \rightarrow \mathbb{P}^n$ imply an inequality. 

\begin{proposition} \label{prop:ineq}
    If $L$ has rank $d+1$, then we have $(-1)^d \cdot \chi(X) \leq \deg {\cal R}_L$.
\end{proposition}
Example \ref{ex:boundarysolution} below shows that the inequality can be strict. Before stating it, we recall a result by Proudfoot and Speyer~\cite{Proudfoot2006} on the defining equations of ${\cal R}_L$. The \emph{circuits} of $M(L)$ are the minimal dependent sets. For each circuit $C \subset \{0, \ldots, n\}$, there is a unique linear relation between the functions $\ell_0, \ldots, \ell_n$, given by $\sum_{i \in C} \alpha_{C,i} \, \ell_i(x) \, = \, 0$.
The coefficient vectors $(\alpha_{C,i})_{i \in C}$ are defined up to scaling. We use these vectors to define one polynomial for each circuit: 
\[ f_C \, = \, \sum_{i \in C} \alpha_{C,i} \, \prod_{j \in C \setminus \{ i \}} \, y_j .\]
The following theorem is crucial in much of what follows, see \cite[Theorem 4]{Proudfoot2006}.
\begin{theorem} \label{thm:proudfoot-speyer}
    The polynomials $\{f_C \, : \, C \text{ is a circuit of $M(L)$} \}$ form a universal Gr\"obner basis for the vanishing ideal $I({\cal R}_L)$ of the reciprocal linear space ${\cal R}_L$.
\end{theorem}

\begin{example} \label{ex:boundarysolution}
    The inequality in Proposition \ref{prop:ineq} can be strict.  The variety ${\cal R}_L$~for
    \[ L  \, = \, \begin{pmatrix}
        1 & 0 & 0 & 0 \\ 
        1 & 1 & 0 & 1 \\ 
        0 & 0 & 1 & 1
    \end{pmatrix}\]
    is a quadratic surface in~$\mathbb{P}^3$. Its defining equation is found via Theorem \ref{thm:proudfoot-speyer}: $ y_2y_3 + y_1y_3-y_1y_2 =   0$.
    Indeed, $M(L)$ has a unique circuit $\{1,2,3\}$  consisting of the last three columns. The complement of the arrangement ${\cal A}$ in $\mathbb{R}^2$ has one bounded box. By Theorem \ref{thm:varchenko}, the Euler characteristic of $X = \mathbb{C}^2 \setminus {\cal A}$ is one.
\end{example}
We conclude the section with a note on algebraic statistics. The probability simplex of dimension $n$ is the set $\Delta_n = \{ p \in \mathbb{R}^{n+1}_+ \, : \, \sum_{i = 0}^n p_i = 1 \}$. This contains all probability distributions for a discrete random variable with $n+1$ states. A \emph{linear model} is the intersection of $\Delta_n$ with an affine subspace in $\mathbb{R}^{n+1}$. That affine-linear space is parametrized by affine-linear functions $\ell_i$: $x \mapsto (\ell_0(x), \ldots, \ell_n(x))$. Since $(\ell_0, \ldots, \ell_n)$ represents a probability distribution, we impose $\sum_{i=0}^n \ell_i(x) = 1$. A central problem in statistics is the following:
\begin{center}
    \emph{Let ${\cal M} \subset \Delta_n$ be a statistical model. Suppose that, in an experiment, state $i$ of our random variable is observed a total number of $u_i$ times. What is the distribution in ${\cal M}$ that best explains the data $u = (u_0, \ldots, u_n)$?}
\end{center}
The answer from \emph{maximum likelihood inference} (MLI) is to maximize the log-likelihood function. The \emph{maximum likelihood estimate} (MLE) is the maximizer on ${\cal M}$. In our context, the log-likelihood function is the scattering potential ${\cal L}_u$ from \eqref{eq:loglikelihood_intro}. The MLE is among its complex critical points. The generic number of complex critical points, called the maximum likelihood degree (ML degree) of ${\cal M}$, governs the algebraic complexity of MLI. In our context, the ML degree is $(-1)^d \cdot \chi(X)$ by Theorem \ref{thm:orlik-terao}, where $X = \mathbb{C}^d \setminus {\cal A}$. The connection with particle scattering was explored in \cite{sturmfels2021likelihood}. For more on linear models, see \cite[Section 1]{HuhSturmfels2014}.

\begin{example} After scaling $\ell_0,\ell_1,\ell_2,\ell_3$ from Example \ref{ex:intro} by $\frac{3}{4},\frac{3}{4},\frac{1}{4}, \frac{1}{4}$ respectively, they sum to one.
    %\[L \, = \, \begin{pmatrix}
      %  0 & 0 & \frac{1}{2} & \frac{1}{2} \\ \frac{3}{4} & 0 & \frac{-1}{4} & \frac{-1}{2} \\ 0 & \frac{3}{4} & \frac{-1}{2} & \frac{-1}{4}
    %\end{pmatrix}.\]
    This scaling changes neither the arrangement ${\cal A}$ nor the scattering equations. Our linear statistical model is the intersection of the 3-dimensional probability simplex with the affine-linear space parametrized by $(\ell_0(x), \ell_1(x), \ell_2(x), \ell_3(x))$. This is the quadrilateral shaded in blue in  Figure \ref{fig:4lines}. It is defined by the inequalities $\ell_i \geq 0 \text{ for } i = 0, \ldots, 3$. The ML degree is three, and for $u \in \mathbb{R}^4_+$, the MLE is the unique critical point of ${\cal L}_u$ contained in the blue region (Example \ref{ex:introcontd}). 
\end{example}

\section{Reciprocal versus ML degree} \label{sec:3}
We investigate when the \emph{reciprocal degree} $\deg {\cal R}_L$ differs from the ML degree $(-1)^d \cdot \chi(X)$. Our starting point is the stratification of ${\cal R}_L$ established in \cite[Proposition 5]{Proudfoot2006} and recalled below in Proposition \ref{prop:stratification}. For any subset $I \subseteq \{0, \ldots, n \}$, we define submatrices $L_I$ and  $A_I^\top$ of $L$ and $A^\top$ respectively. They consist of the columns indexed by $I$. We denote the torus orbits of $\mathbb{P}^n$ by $U_I  = \{ x \in \mathbb{P}^n \,: \, x_i \neq 0 \Leftrightarrow i \in I \}$. The submatrix $L_I$ parametrizes a reciprocal linear space ${\cal R}_{L_I} \subset \mathbb{P}^{|I|-1}$ of dimension ${\rm rank}(L_I)-1$. With a slight abuse of notation, we also write ${\cal R}_{L_I}$ for the image of ${\cal R}_{L_I}$ under the inclusion $\mathbb{P}^{|I|-1} \hookrightarrow \mathbb{P}^n$ which identifies $\mathbb{P}^{|I|-1}$ with $\{ x \in \mathbb{P}^n\,:\, x_i = 0, i \notin I\}$. Finally, we write ${\cal R}_{L_I}^\circ = {\cal R}_{L_I} \cap U_I$. 

\begin{proposition} \label{prop:stratification}
    If $I \subseteq \{0, \ldots, n \}$ is not a flat of $M(L)$, then ${\cal R}_L \cap U_I = \emptyset$. If $I$ is a flat of $M(L)$, then ${\cal R}_L \cap U_I = {\cal R}_{L_I}^\circ$. 
\end{proposition}

Recall that a \emph{flat} of the matroid $M(L)$ is a subset $I$ of $\{ 0, \ldots, n\}$ such that the rank of $I \cup \{i \}$ is strictly greater than ${\rm rank}(I)$ for any $i \in \{0, \ldots, n \} \setminus I$. The entire ground set $\{0, \ldots, n \}$ is a flat by convention. As a consequence of Proposition~\ref{prop:stratification}, we have a disjoint union
\begin{equation} \label{eq:sqcup} \mathbb{L}_u \cap {\cal R}_L \, = \, \bigsqcup_{\substack{I \subseteq \{0, \ldots, n \} \\ I \text{ is a flat of $M(L)$}}}  \mathbb{L}_u \cap {\cal R}_{L_I}^\circ .\end{equation}
\begin{remark}
 Notice that the image of the map $\phi$ from the Introduction is contained in the dense stratum: ${\rm im} \, \phi \subset {\cal R}_L^\circ = {\cal R}_{L_{\{0,\ldots, n\}}}^\circ = {\cal R}_L \cap U_{\{0, \ldots, n \}}$. In particular, the solutions to the scattering equations are among the points 
in $\mathbb{L}_u \cap {\cal R}_{L}^\circ$.   
\end{remark}
\begin{lemma} \label{lem:nosolsatinf}
If ${\rm rank}(L) = d+1$ and $u \in \mathbb{C}^{n+1}$ is generic, then the solutions to the scattering equations are in one-to-one correspondence with the intersections of $\mathbb{L}_u$ with the dense stratum of ${\cal R}_L$. In symbols, we have $\mathbb{L}_u \cap {\rm im} \, \phi = \mathbb{L}_u \cap {\cal R}_L^\circ$. 
\end{lemma}

\begin{proof}
    Observe that ${\cal R}_L^\circ \setminus {\rm im} \, \phi$ equals ${\cal R}_{A^\top}^\circ = {\cal R}_{A^\top} \cap U_{\{0, \ldots, n\}}$. These are the reciprocals of the points in the row span of $L$ which have non-zero coordinates and which lie in the span of the last $d$ rows. The intersections of ${\cal R}_{A^\top}^\circ$ with $\mathbb{L}_u$ are in one-to-one correspondence with the solutions to the scattering equations of~$\tilde{\cal A}$, the central hyperplane arrangement in $\mathbb{C}^d$ given by the columns of $A^{T}$. Since ${\rm rank}(L) = d+1$, $A$ has rank $d$ and ${\tilde{\cal A}}$ is essential. The Euler characteristic of $\mathbb{C}^d \setminus {\tilde{\cal A}}$ is zero, so its scattering equations have no solutions by Theorem~\ref{thm:orlik-terao}.
\end{proof}

\begin{example}
    The matrix $L$ from Example \ref{ex:boundarysolution} gives a quadratic surface ${\cal R}_L$~in~$\mathbb{P}^3$. That surface contains the curve ${\cal R}_{A^\top} = \{ y \in {\cal R}_L \, : \, y_0 = y_1 \}$. The open subset ${\cal R}_{A^\top}^\circ \subset {\cal R}_{A^\top}$ intersects the line $\mathbb{L}_u$ if and only if $u_0 + u_1 + u_2 + u_3 = 0$. 
\end{example}

    The assumption ${\rm rank}(L) = d+1$ can be dropped in Lemma \ref{lem:nosolsatinf}. The central arrangement $\tilde{\cal A}$ from our proof might not be essential in that case, but its scattering equations will have no solutions unless $u_0 + u_1 + \cdots + u_n = 0$:
    \[ \textstyle \sum_{j = 1}^d x_j \, \partial_j {\cal L}_u \, = \, \textstyle \sum_{j = 1}^d x_j \sum_{i=0}^n u_i (\partial_j \ell_i)\ell_i^{-1} \, = \, u_0 + u_1 +  \cdots + u_n, \]
    where $\partial_j = \frac{\partial}{\partial x_j}$. The assumption ${\rm rank}(L) = d+1$ is natural in this section: It only makes sense to compare reciprocal and ML degree when ${\rm rank}(A) = {\rm rank}(L)-1$. Then ${\rm rank}(L) = d+1$, after possibly changing coordinates. 

Next, we identify strata ${\cal R}_{L_I}^\circ$ containing ``excess'' intersection points with~$\mathbb{L}_u$. 
For $I \subseteq \{0, \ldots, n \}$, let ${\cal A}_I = V(\prod_{i \in I} \ell_i(x))$ be the subarrangement of hyperplanes indexed by $I$. Let $K_I \subset \mathbb{C}^d$ be the left kernel of the matrix $A_I^\top$ and let $X_I = (\mathbb{C}^d \setminus {\cal A}_I)/K_I$. Note that $X_I$ is a very affine variety isomorphic to the complement of an arrangement of $|I|$ hyperplanes in $\mathbb{C}^{{\rm rank}(A_I^\top)}$.
\begin{theorem} \label{thm:flats}
    If ${\rm rank}(L) = d+1$ and $u \in \mathbb{C}^{n+1}$ is generic, then the intersection $\mathbb{L}_u \cap {\cal R}_L$ consists of finitely many points. For a flat $I$ of $M(L)$ we~have 
    \begin{enumerate}
    \setlength{\itemsep}{0cm}
        \item if ${\rm rank}(A_I^\top) = {\rm rank}(L_I)$, then $\mathbb{L}_u \cap {\cal R}_{L_I}^\circ = \emptyset$, 
        \item if ${\rm rank}(A_I^\top) = {\rm rank}(L_I) -1$, then the set-theoretic intersection $\mathbb{L}_u \cap {\cal R}_{L_I}^\circ$ consists of $(-1)^{{\rm rank}(A_I^\top)} \cdot \chi(X_I)$ many points.
    \end{enumerate}
\end{theorem}
\begin{proof}
    We start with the proof of (i). Let $\hat{A}_I^\top$ be a matrix of size ${\rm rank}(A_I^\top) \times |I|$ with the same row span as $A_I^\top$. Since ${\rm rank}(A_I^\top) = {\rm rank}(L_I)$ we have ${\cal R}_{\hat{A}_I^\top} = {\cal R}_{A_I^\top} = {\cal R}_{L_I}$. The intersection $\mathbb{L}_u \cap {\cal R}_{L_I}^\circ$ can equivalently be expressed as 
    \[ \mathbb{L}_u \cap {\cal R}_{L_I}^\circ  \, = \, \{ y \in {\cal R}_{\hat{A}_I^\top}^\circ \, : \, \hat{A}_I^\top \,  {\rm diag}(u_I) \, y_I = 0 \}. \]
    Like in the proof of Lemma \ref{lem:nosolsatinf}, these are the scattering equations of an essential central arrangement in $\mathbb{C}^{{\rm rank}(A_I^\top)}$ with signed Euler characteristic zero. There are no solutions for generic $u_I$, which proves part (i). 

    For part (ii), let $\hat{L}_I = \left ( \begin{smallmatrix}
        -b_I^\top \\ \hat{A}_I^\top
    \end{smallmatrix}\right)$, with $\hat{A}_I^\top$ as above. We have ${\cal R}_{\hat{L}_I} = {\cal R}_{L_I}$. We find that $\mathbb{L}_u \cap {\cal R}_{L_I}^\circ$ is described by the scattering equations of ${\cal A}_I \subset \mathbb{C}^{{\rm rank}(A_I^\top)}$: 
    \[ \mathbb{L}_u \cap {\cal R}_{L_I}^\circ  \, = \, \{ y \in {\cal R}_{\hat{L}_I}^\circ \, : \, \hat{A}_I^\top \,  {\rm diag}(u_I) \, y_I = 0 \}. \]
    By Theorem \ref{thm:orlik-terao}, this set consists of $(-1)^{{\rm rank}(A_I^\top)} \cdot \chi(X_I)$ points. 
\end{proof}

\begin{corollary} \label{cor:criterion}
    Let ${\rm rank}(L) = d+1$. We have $\mathbb{L}_u \cap {\rm im} \, \phi = \mathbb{L}_u \cap {\cal R}_L$ for generic $u \in \mathbb{C}^{n+1}$ if and only if all flats $I$ of $M(L)$ except $I = \{0,\ldots,n \}$ are such that ${\rm rank}(A_I^\top) = {\rm rank}(L_I)$. That is, under this condition, the solutions to the scattering equations are in one-to-one correspondence with the points in $\mathbb{L}_u \cap {\cal R}_L$.
\end{corollary}

\begin{remark} Geometrically, the flats of $M(L)$ are linear spaces in $\mathbb{P}^{d}$ obtained as intersections of subsets of the $n + 1$ hyperplanes given by $(x_0, \ldots, x_d)^\top \, L  = (A \, x - b \, x_0)^\top = 0$. The criterion in Corollary \ref{cor:criterion} is equivalent to no non-empty flats being contained in the hyperplane at infinity $\{ x_0= 0 \}$.
\end{remark}

\begin{example} \label{ex:uniformLandA}
    The matrix $L$ from Example \ref{ex:intro} gives rise to the rank-three uniform matroid on four elements, and $M(A^\top)$ is uniform of rank $2$. It is easy to verify that if $L$ has rank $d+1$ and $M(A^{\top})$ is uniform, then the criterion of Corollary \ref{cor:criterion} is satisfied. That is, the matroid of a generic $L$ has no flats at infinity. Hence, as observed in Example \ref{ex:intro}, the intersection points of $\mathbb{L}_u$ and ${\cal R}_L$ are in one-to-one correspondence with the solutions to the scattering equations. 
\end{example}

\begin{example}
    The flats of $M(L)$ with $L$ as in Example \ref{ex:boundarysolution} are $\emptyset$, $\{0\}$, $\{1 \}$, $\{2\}$, $\{3\}$, $\{0,1\}$, $\{0,2\}$, $\{0,3\}$, $\{1,2,3\}$ and $\{0,1,2,3\}$. By Theorem \ref{thm:flats}, the only strata of ${\cal R}_L$ contributing to the intersection $\mathbb{L}_u \cap {\cal R}_L$ are those for which ${\rm rank}(A_I^\top) = {\rm rank}(L_I)-1$. As we had observed in Example \ref{ex:boundarysolution}, for generic $u$, the two intersection points are contained in ${\cal R}_{L_{\{0,1,2,3\}}}^\circ$ and in ${\cal R}_{L_{\{0,1\}}}^\circ$. The flat $\{0,1\}$ is the intersection point of $\ell_0 (x) = \ell_1(x) = 0$, which lies at infinity in $\mathbb{P}^2$. 
\end{example}

\begin{cor}
    Let ${\rm rank}(L) = d+1$. The equality $(-1)^d \cdot \chi(X) = \deg {\cal R}_L$ holds if and only if ${\rm rank}(A_I^\top) = {\rm rank}(L_I)$ for each flat $I$ of $M(L)$, except $I = \{0, \ldots, n\}$. 
\end{cor}

\begin{proof}
    If the condition in the corollary is satisfied and $u$ is generic, then $\mathbb{L}_u \cap {\cal R}_L$ consists of $(-1)^d \cdot \chi(X)$ points by Corollary \ref{cor:criterion}. These intersection points have multiplicity one by Theorem \ref{thm:orlik-terao}. The equality $(-1)^d \cdot \chi(X) = \deg {\cal R}_L$ follows from the fact that a transverse intersection of an $(n-d)$-dimensional linear space with a $d$-dimensional algebraic variety consists of its degree many points. 

    If the condition is violated and $u$ is generic, then $\mathbb{L}_u \cap {\cal R}_L$ consists of more than $(-1)^d \cdot \chi(X)$ isolated points. Therefore $\deg {\cal R}_L >(-1)^d \cdot \chi(X)$.
\end{proof}

\section{Proudfoot-Speyer homotopies} \label{sec:4}
This section explains our method for finding all solutions to \eqref{eq:scatteringeqs} numerically. The algorithm is implemented in Julia (\rm{v1.10.5}) using  \texttt{Oscar.jl}~\cite{OSCAR} (\rm{v1.0.4}) and \texttt{HomotopyContinuation.jl} \cite{breiding2018homotopycontinuationjl} (\rm{v2.0}). All code is available at \cite{mathrepo}.

Our main computational tool is \emph{homotopy continuation}. We recall the basic ideas and refer to the textbook \cite{Sommese:Wampler:2005} for more details. Homotopy continuation is a computational paradigm for finding approximate isolated solutions of systems of polynomial equations. It is based on a deformation of the polynomial system at hand, called the \emph{target system}, into another polynomial system, called the \emph{start system}, whose solutions are easy to compute. Concretely, let $F(x) = (f_1(x), \ldots, f_\ell(x)) = 0$ be a system of $\ell$ polynomial equations in $k \leq \ell$ variables $(x_1, \ldots, x_k)$. A homotopy for solving $F(x) = 0$ is a polynomial map
$H(x,t)  = (h_1(x,t), \ldots, h_\ell(x,t)) \colon \CC^k \times [0,1] \longmapsto \CC^\ell$
satisfying
\begin{enumerate}
\setlength{\itemsep}{0cm}
    \item[1.] $H(x,1) = F(x)$.
    \item[2.] The start system $G(x)=H(x,0)=0$ has at least as many regular isolated solutions in $\mathbb{C}^k$ as $F(x)=0$ and they are easy to compute.
    \item[3.] For any $t\in [0,1)$, the system $H(x,t) = 0$ has the same number of regular isolated solutions in $\mathbb{C}^k$ as $G(x)= 0$.
\end{enumerate}
A \emph{regular isolated solution} of $H(x,t) = 0$ for fixed $t$ is a point $x \in \mathbb{C}^k$ at which $H(x,t) = 0$ and the $\ell \times k$ Jacobian matrix $\left(\frac{\partial h_i}{\partial x_j}(x,t)\right)_{ij}$ has rank $k$. The new variable $t$ is called the \emph{continuation parameter}. The task of a homotopy algorithm is to track each solution of the start system $G(x)=0$ along a continuous solution path as $t$ moves from $0$ to $1$. Such a solution path is a parametric curve $x(t)$ satisfying $H(x(t),t) = 0$. Under suitable assumptions, the solutions of $F(x) = 0$ are among the limits of these paths for $t \rightarrow 1$. In practice, tracking the paths numerically comes down to solving Davidenko's differential equation using \emph{predictor-corrector schemes}, see \cite[Section 2.3]{Sommese:Wampler:2005}.
If the start system $G$ has as many regular isolated solutions as $F$ and they all converge to a solution of $F$, then the homotopy $H(x,t)$ is called \emph{optimal}. This is the favorable case in which no path is lost along the way, so that no computational effort is~wasted.

\begin{example} \label{ex:introcontdcontd}
    In Example \ref{ex:intro} we saw that the homotopy $H(y,t)$ given by
    \[ \small \begin{pmatrix}
        u_0y_0 - u_2y_2 - 2u_3y_3, \, \, u_1y_1 - 2u_2y_2 -u_3y_3, \,\,
        y_1y_2y_3 - t^{1} \, y_0 y_2y_3 - t^{2} \, y_0y_1y_3 + t^3 \, y_0y_1y_2
    \end{pmatrix}\]
    is optimal for solving \eqref{eq:scatteringexample}. The three solutions for $t = 0$ are easy to compute: we simply solve three linear systems with $y_1 = 0, y_2 = 0$ and $y_3 = 0$ respectively. 
\end{example}

%From a theoretical point of view, our algorithm is based on 
The homotopy in Example \ref{ex:introcontdcontd} is based on a flat degeneration of $\mathcal{R}_L$ to a union of coordinate subspaces. Recall that a \emph{flat degeneration} of a projective variety $V \subseteq \PP^n$ is a family of varieties~$\mathcal{X}$ together with a flat morphism ${\pi\colon \mathcal{X} \to  \CC^1}$ such that any fiber $\pi^{-1}(t)$ with $t \in \CC^1\setminus \{0\}$ is isomorphic to $V$.  These are called the \emph{general fibers}, and $\pi^{-1}(0)$ is the \emph{special fiber}. Flatness ensures that the special fiber shares many properties with the general fiber. This includes dimension and degree, Hilbert function, Cohen-Macaulayness and normality, see \cite{bruns2022determinants}.
%The flat degenerations we use in this paper are \emph{Gröbner degenerations} \cite[Section~15.8]{Eisenbud}. %More precisely, let~$\CC[\mathcal{R}_L]$ be the coordinate ring of a reciprocal linear space $\mathcal{R}_L$, that can be represented as a factor ring $\CC[y_0,\ldots,y_n]/I(\mathcal{R}_L)$, and 
%The special fiber in such degenerations is the scheme defined by the initial ideal of the vanishing ideal $I(V)$ of $V$. 

Let $I(\mathcal{R}_L) \subset \mathbb{C}[y_0, \ldots, y_n]$ be the vanishing ideal of $\mathcal{R}_L$, as above. 
We consider  the initial ideal $J$ of $I(\mathcal{R}_L)$ with respect to a weight vector $\omega \in \mathbb{Z}^{n+1}$:
$$J \, := \, \mathrm{in}_{\omega}(I(\mathcal{R}_L)) \,  = \, \mathrm{span}_{\CC}\{\mathrm{in}_{\omega}(f) \, : \,  f\in I(\mathcal{R}_L)\}.$$ 
Its variety is $V(J)$. By \cite[Theorem 4]{Proudfoot2006}, for a  generic weight vector $\omega$, $J$ is a square-free monomial ideal and $V(J)$ is a union of coordinate subspaces.
Moreover, $V(J)$ is the special fiber in a flat degeneration of ${\cal R}_L$, as we~now explain.

We extend our polynomial ring to $\CC[y,t] = \CC[y_0,\ldots, y_n,t]$ by adding a continuation parameter $t$. %The torus~$\mathbb{C}^\times$ acts on the polynomial ring~$\CC[y_0,\ldots, y_n]\subseteq \CC[y_0,\ldots, y_n,t]$ as follows. 
Let $f(y) = \sum c_\alpha y^{\alpha} \in \CC[y_0, \ldots, y_n]$ be a polynomial. We define~$\omega(f) := \underset{\alpha, \, c_\alpha\neq 0}{\max} \{ \omega \cdot \alpha\}$ and $f_t^\omega(y,t) =  \sum c_\alpha t^{\omega(f)-\omega \cdot \alpha} y^\alpha \in \mathbb{C}[y,t]$. The~ideal
\[\label{weight_degeneration}
 I(\mathcal{R}_L)_t^{\omega} \, := \,  \langle  \,  f_t^{\omega}\, : \,  f\in I(\mathcal{R}_L) \, \rangle \, \subset \,  \CC[y,t]  \]
defines a family of varieties ${\cal X} = V(I(\mathcal{R}_L)_t^{\omega}) \subset \mathbb{P}^n \times \mathbb{C}$. By \cite[Theorem~15.17]{Eisenbud} this family is flat over $\CC$. It defines the Gröbner degeneration of ${\cal R}_L$ with respect to the weight $\omega$, whose special fiber is $\pi^{-1}(0) = V(J)$. 
\begin{example}
    The degeneration of the cubic surface $\mathcal{R}_L$ from Example \eqref{ex:intro} is a Gröbner degeneration with respect to the weight vector $\omega = (1,2,3,4)$.
\end{example}
We now use the degeneration explained above in a homotopy algorithm for solving the scattering equations.  %The target equations \eqref{eq:embeddedeq} on the reciprocal linear space $\mathcal{R}_L$ form the system
Let $A_0^\top \in \CC^{d \times (n+1)}$ be a matrix with generic entries.
The target system $F(y) = 0$ and the start system $G(y)= 0$ are given by 
\begin{equation*}
    F(y) \,= \,  \begin{pmatrix}
        A^\top \, \mathrm{diag}(u) \, y \\
        (f_C(y))_{C \in {\rm circuits}(M(L))} 
    \end{pmatrix}, \quad  G(y) \, = \,  \begin{pmatrix}
        A_0^\top \,  y \\
        ((f_C)_t^\omega(y,0))_{C \in {\rm circuits}(M(L))} 
    \end{pmatrix}.
    \end{equation*}
Here $\{f_C \colon C \in {\rm circuits}(M(L))\}$ is the universal Gröbner basis of $I(\mathcal{R}_L)$ from Theorem \ref{thm:proudfoot-speyer}. The solutions to $F(y) = 0$ are the points in $\mathbb{L}_u \cap {\cal R}_L$. %We define the start system 
%\begin{equation*}
    %G(y) \, = \,  \begin{pmatrix}
     %   A_0^\top \,  y \\
    %    ({\rm in}_\omega f_C(y))_{C \in {\rm circuits}(M(L))} 
   % \end{pmatrix} \, = \, 0,
%\end{equation*}
 To connect the target system $F(y)$ and the start system $G(y)$, we set up the homotopy
\begin{equation}\label{eq:homotopy}
    H(y,t) \, = \,  \begin{pmatrix}
        (1-t)\, A_0^\top \,  y + t \,  A^\top \, \mathrm{diag}(u) \, y  \\
        ((f_C)^\omega_t(y,t))_{C \in {\rm circuits}(M(L))}     \end{pmatrix}.
\end{equation}
In words, $H(y,t)$ is a combination of a \emph{straight line homotopy} for the linear part of the system and a Gröbner degeneration of the reciprocal linear space~$\mathcal{R}_L$. Algorithm \ref{alg:PShomotopy} summarizes how to use this homotopy to solve the scattering equations numerically. Below, we briefly comment on the steps.%We are ready to present our homotopy algorithm for solving the scaterring equations~\eqref{eq:embeddedeq} on a reciprocal linear space numerically.

 \begin{algorithm}
 \small
\caption{Proudfoot-Speyer homotopy algorithm}
 \label{alg:PShomotopy}
\SetAlgoLined
\KwIn{A matrix $L\in \CC^{(d+1)\times (n+1)}$ of rank $d+1$ representing a hyperplane arrangement ${\cal A}$ and a generic vector $u \in \CC^{n+1}$.}
\vspace{0.2cm}
\KwOut{Solutions to the equations \eqref{eq:scatteringeqs}.}
%\begin{enumerate}
\vspace{0.2cm}
    %\item[A.] Precomputation
    
    1. Choose a generic vector $\omega \in \ZZ^{n+1}$ and generic $A_0^\top \in \CC^{d \times (n+1)}$.
    
    2. Compute the polynomials $\{f_C \colon C \text{ is a circuit of } M(L)\}$.

    3. Find a minimal prime decomposition of the ideal $J = \mathrm{in}_\omega(I(\mathcal{R}_L))$.

    %\vspace{-0.2cm}
    %\item[B.] Solving linear systems
    
    4. For each irreducible component $Y_i$ of $V(J)$, find the unique  solution $y\in Y_i$ of the linear system $A_0^\top \, y = 0$.

    %\vspace{-0.2cm}
    %\item[C.] Homotopy

    5. For each of the solutions from step 4, trace the homotopy \eqref{eq:homotopy} along a smooth path in $\CC^1$ from $t=0$ to $t=1$.
    
    6. Return the inverse image under $\phi$ of the solutions from step 5 that 
    have only non-zero  coordinates.
    \vspace{0.2cm}
%\end{enumerate}
\end{algorithm}
The genericity condition for $\omega$ in step 1 is that the weight $\omega$ should define a linear order on the ground set $\{0,\ldots,n\}$ of the matroid $M(L)$. That is, all its entries should be distinct. The genericity condition for the matrix $A_0^\top$ is that $A_0^\top \, y = 0$ defines a linear subspace of codimension $d$ which cuts the variety $V(J)$ in $\deg V(J) = \deg \mathcal{R}_L$ many points. In our code, $A_0^\top$ can optionally be inputted by the~user. We have seen in Example \ref{ex:intro} that one can sometimes pick $A_0^\top = A^\top {\rm diag}(u)$, so that the first $d$ equations in $H(y,t)$ do not involve $t$. The matrix $A^\top {\rm diag}(u)$ might not satisfy our genericity assumption, see Remark \ref{rem:badstart}.

In step 2, we compute the circuits of $M(L)$ to find the universal Gröbner basis from Theorem \ref{thm:proudfoot-speyer}. For step 3, we find broken circuits with respect to the weight vector $\omega$ that generate the ideal $J$ following~\cite{Proudfoot2006}. The minimal prime decomposition of $J$ is then computed using only the combinatorics of $M(L)$. Let $\omega$ be a linear order on $\{0,\ldots,n\}$. Recall from \cite{Proudfoot2006} that a \emph{broken circuit} of $M(L)$ is obtained from a circuit of $M(L)$ by deleting the element with the smallest $\omega$-weight.
Let $M^{\omega}(L)$ be the matroid on $\{0,\ldots,n\}$ whose circuits are the minimal broken circuits of $M(L)$ with respect to inclusion.
\begin{proposition}\label{prop:reciprocal_degree}
    The reciprocal degree $\deg \mathcal{R}_L$ equals the number of bases of the matroid~$M^{\omega}(L)$. The minimal primes of the ideal $J$ are $\langle y_i \colon i \in B^c\rangle $, where $B \subset \{0,\ldots,n\}$ runs over all bases of $M^{\omega}(L)$ and $B^c = \{0,\ldots,n\}\setminus B$. 
\end{proposition}
\begin{proof}
    In \cite{Proudfoot2006}, it was shown that $\mathbb{C}[\mathcal{R}_L]$ flatly degenerates to the Stanley-Reisner ring of the broken circuits simplicial complex $\mathrm{bc}_{\omega}(L)$ on $\{0,\ldots,n\}$. Its faces are subsets of $\{0,\ldots,n\}$ that do not contain any broken circuit. The degree of $\mathcal{R}_L$ is the number of facets of $\mathrm{bc}_{\omega}(L)$, which is the number of maximal subsets of the ground set that do not contain any broken circuit. By construction of the matroid $M^{\omega}(L)$, these are precisely the bases of $M^{\omega}(L)$.
        %To conclude, we recall that 
        The facet complements generate the minimal prime components of the  Stanley-Reisner ideal.
\end{proof}
The fact that the initial monomial ideal $J$ is squarefree implies that the start system has only regular isolated solutions. Since our algorithm relies heavily on results from \cite{Proudfoot2006}, we chose the name \emph{Proudfoot-Speyer homotopy}.

%\begin{remark}
  %  The set of broken circuits of $M(L)$ might not satisfy the minimality condition on circuits. %That is, one broken circuit can be a subset of another broken circuit. In this case, 
%    To define the matroid $M^{\omega}(L)$, we must consider only the broken circuits that minimally generate the ideal $J$. 
%\end{remark}

\begin{theorem} \label{thm:optimal}
    Let ${\rm rank}(L) = d+1$ and suppose that all flats $I$ of $M(L)$ except $I = \{0,\ldots,n \}$ are such that ${\rm rank}(A_I^\top) = {\rm rank}(L_I)$. For generic $u \in \mathbb{C}^{n+1}$, the Proudfoot-Speyer homotopy from Algorithm \ref{alg:PShomotopy} is optimal for solving \eqref{eq:scatteringeqs}, meaning that the number of homotopy paths equals the number of solutions. 
\end{theorem}
\begin{proof}
    The number of homotopy paths in a Proudfoot-Speyer homotopy equals the degree of ${\cal R}_L$. By Corollary \ref{cor:criterion}, the conditions in the theorem imply that this is also the number of solutions to the scattering equations.
\end{proof}

Notice that Theorem \ref{thm:optimal} implies Theorem \ref{thm:optimal_intro} (see Example \ref{ex:uniformLandA}). 
Our Julia package \texttt{ProudfootSpeyerHomotopy} \cite{mathrepo} implements Algorithm \ref{alg:PShomotopy}. %The code and a tutorial on how to use its main functionalities are found at \cite{mathrepo}.
\begin{comment}
We illustrate our main function \texttt{solve}\_\texttt{PS} on the running Examples \ref{ex:intro},~\ref{ex:boundarysolution}.
 \begin{minted}{julia}
using ProudfootSpeyerHomotopy #loading the package
L1 =  [0 0 2 2; 1 0 -1 -2; 0 1 -2 -1]
u1 = randn(ComplexF64,size(L1,2)) #input a generic data vector u
solve_PS(L1, u1)

L2 =  [1 0 0 0; 1 1 0 1; 0 0 1 1]
u2 = randn(ComplexF64,size(L2,2))  
solve_PS(L2, u2)
    \end{minted}
The output in line 4 is a 3-element vector of solutions.  Reciprocal degree 3 and ML degree 3 are printed by the function. The output in line 8 is a 1-element vector of solutions. The function prints reciprocal and ML degree, that are 2 and 1 respectively.
\end{comment}

\begin{example}
   In Table \ref{tab:generic_arrangement} we report timings for generic hyperplane arrangements with several values for $d$ and $n$. Our homotopy method is optimal for such arrangements, see Theorem \ref{thm:optimal}. In our experiment, the entries of the matrix $L$ are random uniformly distributed integer numbers in the interval $[-20,20]$ and the parameters $u$ are random complex numbers drawn from a standard normal distribution. The number of solutions in each case is $\binom{n}{d}$.
\begin{table}[ht]
\centering
\footnotesize
\begin{tabular}{c|cccccc}
\diagbox{$d$}{$n$}  & 6         & 7         & 8       & 9        & 10        & 11 \\ \hline 
2& $0.07s$ & $0.19s$ & $0.50s$ & $1.11s$ & $2.90s$ & $6s$ \\ 
 3              & $0.076s$  & $0.29s$   & $1.20s$ & $3.90s$  & $13.85s$  & $41.90s$ \\ 
 4              & $0.041s$  & $0.21s$   & $1.26s$ & $13.42s$ & $30.16s$  & $132s$ \\ 
 5              & $0.017s$  & $0.10s$   & $0.54s$ & $7.81s$  & $41.15s$  & $194s$ \\ 
 6              &           & $0.02s$   & $0.54$  & $2.78s$  & $23.50s$  & $236s$ \\ 
 7              &           &           & $0.03s$ & $2.64s$  & $9.94s$   & $124s$ \\ 
 8              &           &           &         & $0.04s$  & $9.74s$   & $59s$ \\ 
\end{tabular}
    \caption{Timing results (in seconds) for generic hyperplane arrangements.} \label{tab:generic_arrangement}
\end{table}
\end{example}

\vspace{-0.5cm}

\section{Scattering equations on $\mathcal{M}_{0,m}$} \label{sec:5}
%In this section we provide the connection to theoretical particle physics. 
In this section we focus on $X \simeq \mathcal{M}_{0,m}$, the configuration space of $m$ distinct points on the projective line $\PP^1$. %Using the $\mathrm{PGL}(2) \times (\CC^{\times})^m$ free action (see \cite[Lemma 1.2]{lam2024moduli}), we may present 
A point in $\mathcal{M}_{0,m}$ is represented as a $2\times m$~matrix
\begin{equation}\label{point_of_M0m}
    \begin{pmatrix}
1 & 1 & 1 & \ldots & 1 & 0\\
0 & 1 & x_1 & \ldots & x_{m-3} & 1
\end{pmatrix}
\end{equation}
whose $2 \times 2$-minors $p_{ij}(x)$ are non-zero. 
The $i$-th column represents homogeneous coordinates of a point $\sigma_i \in \PP^1$ and imposing that the minors are non-zero means $\sigma_i \neq \sigma_j$ for $i \neq j$.  The \emph{CHY (Cachazo-He-Yuan) scattering equations}~are 
\begin{equation}\label{eq: CHY}
 \frac{\partial \mathcal{L}_s}{\partial x_1} \, = \ldots = \, \frac{\partial \mathcal{L}_s}{\partial x_{m-3}} \, = \, 0, \quad \text{ where } \quad \mathcal{L}_s \, = \,  \log \prod_{i<j} p_{ij}(x)^{s_{ij}}.
 \end{equation}
The exponents $s_{ij}$ are called \emph{Mandelstam invariants} in physics. They encode the momenta of $m$ particles involved in a scattering process. The columns of the $2\times m$-matrix \eqref{point_of_M0m} are indexed by these particles.  %They are coming from a \emph{quantum field theory} and should satisfy a symmetry condition: $s_{ii} = 0$ and $s_{ij} = s_{ji}$ and a momentum conservation $\sum_{j=1}^m s_{ij} = 0$. The log-likelihood function $\mathcal{L}_s$ plays the role of the \emph{scattering potential}. 
The \emph{CHY amplitude} of the scattering process is a global residue over the solutions of~\eqref{eq: CHY}. It is a rational function in the Mandelstam invariants $s_{ij}$.
This CHY formalism motivates the importance of solving scattering equations in theoretical particle physics \cite{cachazo2014scattering,lam2024moduli}.

The above discussion models $\mathcal{M}_{0,m}$ as a hyperplane arrangement complement. The arrangement is $\mathcal{A}_m = V(\prod_{i<j} p_{ij}(x))$ in $\CC^{m-3}$. There are $(m-3)!$ bounded regions in the corresponding real arrangement complement in~$\RR^{m-3}$ \cite[Proposition 1]{sturmfels2021likelihood}. Theorem \ref{thm:varchenko} tells us that \eqref{eq: CHY} has $(m-3)!$ solutions.

Let $L_m$ be the matrix associated to the hyperplane arrangement~$\mathcal{A}_m$. Since the minors $p_{ij}$ are of the form $x_i$, $x_i - 1$ or $x_j - x_i$ for $j>i$, we can write $L_m$ as
\begin{comment}
\begin{equation}\label{matrix: Lm}
\newcommand{\?}[1]{\multicolumn{1}{r|}{#1}}
L_m = 
%\begin{footnotesize} 
\NiceMatrixOptions{xdots/shorten = 0.6 em}
\begin{pNiceArray}[small]{rrrrrrrrrrrrcrr}
%\cline{2-6} \cline{8-13}
\?{}& 0 & \?{-1} & 0 & -1 & \?0 & \?{\cdots  \hspace*{0.05cm}} & 0 & -1 & 0 & 0 &\cdots  & \?0 & \\
\?{}&1 & \?{1}  & 0 & 0  & \?{-1} & \?{ \cdots  \hspace*{0.05cm}} & 0 & 0 & -1 & 0 & \cdots & \?0 & \\
\cline{2-3}
&0 & \?{0}  & 1 & 1  & \?{1} & \?{\cdots  \hspace*{0.05cm}} & 0 & 0 & 0 &  -1 & \cdots & \?0& \\
\cline{4-6}

& \vdots & \vdots & \vdots & \vdots & \vdots & \?{\ddots}  & \vdots & \vdots & \vdots & \vdots & \ddots & \?{\vdots} \\

% &\vdots \hspace*{0.05cm} & \vdots \hspace*{0.05cm}  & \vdots \hspace*{0.05cm} & \vdots \hspace*{0.05cm}  & \vdots \hspace*{0.05cm} & \?{\hspace{2cm}} & \vdots \hspace*{0.05cm} & \vdots \hspace*{0.05cm} & \vdots \hspace*{0.05cm} &   \vdots \hspace*{0.05cm} &  & \?{\vdots \hspace*{0.05cm}} & \\

&0 & 0  & 0 & 0  & 0 & \?{\cdots  \hspace*{0.05cm}} & 0 & 0 & 0 &  0 & \cdots & \?{-1} & \\
&0 & 0  & 0 & 0  & 0 & \?{\cdots  \hspace*{0.05cm}} & 1 & 1 & 1 &  1 & \cdots & \?{1} & \\
\cline{8-13}
%\CodeAfter \line[shorten = 0.5 em]{3-11}{5-13} 
%\line[shorten = 1 em]{3-6}{5-8}
\end{pNiceArray}.   
%\end{footnotesize}
\end{equation}
\end{comment}
\setcounter{MaxMatrixCols}{20}
\begin{equation} \label{matrix: Lm}
L_m \, = \,\begin{tiny} \begin{pmatrix}
    0 & -1 & \vrule & 0 & -1 & 0 & \vrule & \cdots & \vrule & 0 & -1 & 0 & 0 & \cdots & 0 \\
    1 & 1 & \vrule & 0 & 0 & -1 & \vrule & \cdots & \vrule & 0 & 0 & -1 & 0 & \cdots & 0 \\
    0 & 0 & \vrule & 1 & 1 & 1 & \vrule & \cdots & \vrule & 0 & 0 & 0 & -1 & \cdots & 0 \\
    \vdots & \vdots&  \vrule & \vdots& \vdots & \vdots &  \vrule & \cdots &  \vrule & \vdots & \vdots & \vdots & \vdots & \ddots & \vdots \\
    0 & 0 & \vrule & 0 & 0 & 0 & \vrule & \cdots & \vrule & 0 & 0 & 0 & 0 & \cdots & -1 \\
    0 & 0 & \vrule & 0 & 0 & 0 & \vrule & \cdots & \vrule & 1 & 1 & 1 & 1 & \cdots & 1
\end{pmatrix} \end{tiny}.\end{equation}

In words, $L_m$ is a $(m-2) \times \frac{m(m-3)}{2}$ matrix of rank $m-2$ that consists of $m-3$ rectangular blocks of sizes $(m-2) \times k$ for $k = 2, \ldots, m-2$. The top square $k\times k$ submatrices of these blocks have $-1$ on the first upper diagonal, and their $k$-th row consists of ones. The parameters $d, n$ are $d = m-3$ and $n = \frac{m(m-3)}{2}-1$.

In the spirit of previous sections, we translate \eqref{eq: CHY} into linear equations on the $(m-3)$-dimensional reciprocal linear space 
%$(m-3)$-fold 
${\cal R}_{L_m}$. We intersect ${\cal R}_{L_m}$ with the linear space $\mathbb{L}_s =\{ y \in \mathbb{P}^n \,:\, A_m \,  {\rm diag}(s) \, y = 0\}$, where $A_m$ consists of the last $m-3$ rows of $L_m$ and $s$ is a vector of Mandelstam invariants $s_{ij}$ in a suitable order. The degree $\deg {\cal R}_{L_m}$ gives the expected number of intersection points.  

\begin{proposition} \label{prop:recdegM0m}
    The reciprocal degree of ${\cal M}_{0,m}$ is $\mathrm{deg}\, \mathcal{R}_{L_m} = (m-3)(m-3)!$.
\end{proposition}

To prove Proposition \ref{prop:recdegM0m}, we study the matroid $M(L_m)$ in more~detail.
\begin{lemma}\label{lem: circuitsstruct}
    Let $\mathcal{C}$ be a circuit of $M(L_m)$. Then $\mathcal{C}$ contains at most 2 elements from each of the $m-3$ block columns of $L_m$ as in \eqref{matrix: Lm}.
\end{lemma}
\begin{proof}
Assume first that $\mathcal{C}$ contains at least three vectors of the form $\omega_0 = e_k$, ${\omega_i = e_k - e_i}$, and $\omega_j = e_k - e_j$ from the $k$-th block, where $i < j$. Since $\mathcal{C}$ is a circuit, $\mathcal{C}\setminus \{\omega_h\}$ is independent for any $h = 0, i, j$. Additionally, $\mathcal{C}$ must include other vectors with non-zero entries in rows $i$ and $j$ that are different from~$\omega_i$,$\omega_j$.
The triples ${\{\omega_i, \omega_j, e_j - e_i\}}$ and $\{\omega_{0}, \omega_h, e_h\}$ are 3-circuits of $M(L_m)$, so no other vector in~$\mathcal{C}$ can be of the form $e_j - e_i$ or $e_h$, where $h = i, j$. Therefore, $\mathcal{C}$ must include at least two other vectors of the form $\pm(e_j - e_{\ell_1}), \pm(e_i - e_{r_1})$, where $\ell_1 \neq r_1 \notin \{i, j, k\}$  to prevent a 4-circuit $\{\pm(e_j - e_{\ell_1}), \pm(e_i - e_{r_1}), \omega_i, \omega_j\}$. However, the span of the vectors in proper subsets of ${\cal C}$ now contains vectors of the form ${\pm(e_p - e_q)}$ for any pair $p, q \in \{i, j, k, \ell_1, r_1\}$.
By a similar argument, to eliminate non-zero entries in rows $\ell_1$ and $r_1$, we need at least two more vectors ${\pm(e_{\ell_1} - e_{\ell_2})}$, ${\pm(e_{r_1} - e_{r_2})}$, where $\ell_2 \neq r_2$ and $\ell_2, r_2 \notin \{i, j, k, \ell_1, r_1\}$. Iterating this process leads to a contradiction: the matrix $L_m$ has finitely many rows, so at some step~$t$, either $\ell_t$ or $r_t$ belongs to $\{i, j, k, \ell_1, r_1, \ldots, \ell_{t-1}, r_{t-1}\}$, and a proper subset of $\mathcal{C}$ is linearly dependent.

If $\omega_0 = e_k - e_s$ for some $s < i$, then to cancel the non-zero entries in rows $s, i, j$ we need to add three distinct vectors, and at most one of them can be of the form $e_h$, while the other two must have non-zero entries in new distinct rows. Thus, at every step, we introduce at least two additional rows of the matrix $L_m$, which again leads us to a contradiction.
\end{proof}
\begin{comment}
 
    Assume that we have at least 3 vectors $\omega_s = e_k - e_s$, $\omega_i = e_k - e_i$ and $\omega_j= e_k - e_j$ from the $k$-th block, with $s < i < j$. Since $\mathcal{C}$ is a minimal dependent set it contains vectors with a non-zero entry in rows $s,i,j$ different from $\omega_s,\omega_i,\omega_j$. Moreover, $\mathcal{C}\setminus\{\omega_h\}$ is independent for every $h\in \{s,i,j\}$, so $\mathcal{C}$ does not contain any vector in $\{\omega_i-\omega_s, \omega_j-\omega_i, \omega_j-\omega_s \} $ and it has no pair of columns with the same non-zero entries. It also does not contain any pair of columns $\{e_h- e_\ell, e_h-e_{\ell'}\}$ with $\ell, \ell' \in \{s,i,j\}$. Hence, for $\mathcal{C}$ to be a circuit, it has to contain at least three more elements with a non-zero entry in rows $i,j$ or $s$ respectively, and at least two of these three elements have another nonzero entry in distinct rows $\ell, r \notin \{s,i,j,k\}$. Thus, the span of the columns in ${\cal C}$ contains $e_r-e_\ell, e_k-e_r, e_k-e_\ell$ and $\pm(e_r-e_s), \pm(e_j-e_r), \pm(e_i-e_r), \pm(e_\ell-e_s),\pm(e_j-e_\ell),\pm(e_i-e_\ell)$. We can now apply the same argument to see that, in order to cancel the non-zero entries in rows $\ell, r$, we must find more columns in ${\cal C}$ with nonzero entries in the complement of $\{\ell,r,s,i,j,k\}$. Iterating this process leads to a contradiction.   
\end{comment}

\begin{proof}[Proof of Proposition \ref{prop:recdegM0m}]
  By Proposition \ref{prop:reciprocal_degree}, the reciprocal degree $\mathrm{deg}\mathcal{R}_{L_m}$ is equal to the number of bases of the matroid $M^{\omega}(L_m)$, where $\omega$ is \emph{any} linear order on the ground set $\{0, \ldots, n\}$. Let us choose the order $\omega = (n, \ldots, 1,0)$. That is, for any subset of the ground set, the largest index is $\omega$-minimal.
  
  We begin by describing the $\omega$-broken circuits of the matroid $M(L_m)$. The 3-circuits of $M(L_m)$ are $\{e_j - e_i, e_k - e_i, e_k - e_j\}$ and $\{e_i, e_j, e_j - e_i\}$ for ${i < j < k}$. The corresponding $\omega$-broken circuits $\{e_j - e_i, e_k - e_i\}$ and $\{e_i, e_j\}$ for ${i < j < k}$ are 2-circuits of $M^{\omega}(L_m)$. %In simple terms, these are either pairs of standard basis vectors from the matrix $L_m$ or pairs of columns that have $-1$ in the same row. 
We prove by induction that any other $\omega$-broken circuit contains a 2-broken circuit. Assume the claim holds for all $(r-1)$-broken circuits. A circuit of columns $v_{i_1}, \ldots, v_{i_{r+1}}$ of $L_m$ with $i_1 < \ldots < i_{r+1}$ gives
\begin{equation}\label{eq: lincomb_circuit}
    v_{i_{r+1}} = \lambda_1 v_{i_1} + \ldots + \lambda_{r} v_{i_{r}}
\end{equation}
with $\mathcal{BC} = \{v_{i_1}, \ldots, v_{i_{r}}\}$ as an $\omega$-broken circuit. Suppose the last vector $v_{i_{r+1}}$ lies in the $k$-th block, then its $k$-th entry is 1. This can only be cancelled if there is another vector from the $k$-th block. By Lemma \ref{lem: circuitsstruct}, this implies that there are exactly 2 vectors $v_{i_{r+1}}, v_{i_{r}}$ from the $k$-th block. Thus, $v_{i_r}$ appears in \eqref{eq: lincomb_circuit} with the coefficient $\lambda_r = 1$. Moreover, $v_{i_{r+1}} \neq e_k$, since it cannot be the first column of the block and we have $v_{i_{r+1}} = e_k - e_j$ for some $j<k$. 

We consider two cases. First, if $v_{i_r} = e_k - e_s$ for $s < j < k$, then the vector $e_j - e_s = -(v_{i_{r+1}} - v_{i_{r}})$ lies in the span of $\mathcal{BC}\setminus \{v_{i_{r}}\}$. Therefore there is a circuit $\mathcal{C}$ of size $|\mathcal{C}| \leq r$ in $\mathcal{BC}\setminus \{v_{i_{r}}\} \cup \{e_j - e_s\}$. By the induction hypothesis, the corresponding $\omega$-broken circuit contains some $2$-broken circuit. Then the whole set ${\mathcal{BC}\setminus \{v_{i_{r}}\} \cup \{e_j - e_s\}}$ contains this $2$-broken circuit. If this has the form $\{e_i, e_j\}$, then it lies in the set ${\mathcal{BC}\setminus \{v_{i_{r}}\}}$ and thus in $\mathcal{BC}$. If it has the form $\{e_q - e_p, e_t - e_p\}$, then either it lies in ${\mathcal{BC}\setminus \{v_{i_{r}}\}}$ or one of its elements is $e_j - e_s$. But then $p = s$ and $\mathcal{BC}$ contains a 2-broken circuit $\{e_q - e_s, v_{i_{r}} = e_k - e_s\}$.

In the second case when $v_{i_r} = e_k$, we apply a similar argument for the dependent set $\mathcal{BC}\setminus \{v_{i_{r}}\} \cup \{e_j\}$.

\begin{comment}
we take a 3-circuit of the form $\mathcal{C}_2 = \{e_j - e_s, v_{i_{l}} = e_k - e_s, v_{i_{r+1}} = e_k - e_j\}$. Since $v_{i_{r+1}} \in \mathcal{C}_1 \cap \mathcal{C}_2$, the set $\mathcal{C}_1 \cup \mathcal{C}_2 \setminus \{ v_{i_{r+1}}\} = \{\}$ 

  To prove this claim, assume that the columns $v_1, \ldots, v_k$ of $L_m$ form a circuit of the matroid $M(L_m)$. Then, $v_k = \lambda_1 v_1 + \ldots + \lambda_{k-1} v_{k-1}$, where all $\lambda$ coefficients are non-zero, and $v_1, \ldots, v_{k-1}$ form an $\omega$-broken circuit. If none of the vectors $v_1, \ldots, v_{k-1}$ are of the form $e_j - e_i$, then they must all be standard basis vectors and thus contain a 2-broken circuit of the form $\{e_i, e_j\}$. Otherwise, there must be some $v_l = e_j - e_i$ for $l < k$. By the structure of the matrix $L_m$, the $i$-th entry of the vector $v_k$ is zero \textcolor{blue}{I don't see why}. Since $v_k$ is a linear combination of $v_1, \ldots, v_{k-1}$, we must have $v_r = e_k - e_i$ for $r < k$ for some $r \neq l$.
\end{comment}
  The matroid $M^{\omega}(L_m)$ is defined by circuits that correspond to the indices of the parallel pairs $\{e_j - e_i, e_k-e_i\}$ and $\{e_i, e_j\}$ for $i < j <k$. This matroid can be represented by a matrix $L^\omega= (\begin{smallmatrix}
      e_1 & e_2& \vrule & e_1 & e_2 & e_3 & \vrule & \cdots & \vrule & e_1 & \cdots & e_{m-2} 
  \end{smallmatrix})$
  %
  %\NiceMatrixOptions{xdots/shorten = 0.6 em}
  %\begin{footnotesize}
  %\begin{pNiceArray}{c|cc|ccc|c|ccc|c}
 %& e_1 & e_2 & e_1 & e_2 & e_3 & \cdots & e_1 & \cdots & e_{m-2} &
 % \end{pNiceArray}.
 % \end{footnotesize}
 % $
with the same block structure as in \eqref{matrix: Lm}, but its blocks are identity matrices. The bases of $M^{\omega}(L_m)$ consist of standard basis sets of the form $\{e_1, \ldots, e_{m-2}\}$. There are $(m-3)(m-3)!$ ways to select such a basis from the columns of $L^\omega$.
\end{proof}

As observed above, the ML degree of ${\cal M}_{0,m}$ is $(m-3)!$. Proposition \ref{prop:recdegM0m} says that its reciprocal degree is $(m-3)$ times larger. This means that for generic $s$, there are $(m-4)(m-3)!$ solutions on the boundary ${\cal R}_{L_m} \setminus {\cal R}_{L_m}^\circ$. We study these boundary solutions using Theorem \ref{thm:flats}. We say that a flat of $M(L_m)$ is \emph{of type (ii)} if it satisfies the condition \rm{(ii)} from Theorem~\ref{thm:flats}. A submatrix of $L_m$ is said to be \emph{equivalent to} $L_{m-r}$ if it is equal to $ L_{m-r} $ after deleting $r$ zero rows.

\begin{proposition}\label{prop: number_of_deg_flats}
   For each $ r = 0, \ldots, m-4 $, there are exactly $ \binom{m-3}{r} $ type \rm{(ii)} flats of~$ M(L_m) $ whose corresponding submatrix of $ L_m $ is equivalent to $L_{m-r}$. 
\end{proposition}

\begin{example}\label{ex: submatricesLm}
The matrix $ L_6 $ contains exactly $ \binom{6-3}{1} = 3 $ submatrices that are equivalent to $ L_5 $, such that removing the first row causes the rank to drop. These submatrices are highlighted in yellow and correspond to the flats $ \{0,1,2,3,4\} $, $ \{2,3,5,6,8\} $, and $ \{0,1,5,6,7\} $ of $M(L_6)$, respectively.
    \begin{equation*}
    \small
 \begin{pNiceArray}[small]{rrrrrrrrr}
\CodeBefore
 \cellcolor[HTML, opacity = 0.7]{FFFF88}{1-1, 1-2, 1-3, 1-4, 1-5,  2-1,2-2,2-3, 2-4, 2-5, 3-1, 3-2, 3-3, 3-4, 3-5}
\Body
 {\bf \color{red}0} & {\bf \color{red}-1} & {\bf \color{teal}0} & {\bf \color{teal}-1} & 0 & 0 & -1 & 0 & 0 \\
{\bf \color{red}1} & {\bf \color{red}1}  & 0 & 0  & -1 & 0 & 0 & -1 & 0 \\
0 & 0  & {\bf \color{teal}1} & {\bf \color{teal}1}  & 1 & 0 & 0 & 0 &  -1 \\
0 & 0   & 0  & 0   & 0 & 1 & 1 & 1 & 1  
\end{pNiceArray} ,   \;
\begin{pNiceArray}[small]{rrrrrrrrr}
\CodeBefore
 \cellcolor[HTML, opacity = 0.7]{FFFF88}{1-3, 1-4, 1-6, 1-7, 1-9, 3-9, 3-3, 3-4, 3-6, 3-7, 3-9, 4-3, 4-4, 4-6, 4-7, 4-9}
\Body
0 & -1 & {\bf \color{teal}0} & {\bf \color{teal}-1} & 0 & {\bf \color{blue}0} & {\bf \color{blue}-1} & 0 & 0 \\
1 & 1  & 0 & 0  & -1 & 0 & 0 & -1 & 0 \\
0 & 0  & {\bf \color{teal}1} & {\bf \color{teal}1}  & 1 & 0 & 0 & 0 &  -1 \\
0 & 0   & 0  & 0   & 0 & {\bf \color{blue}1} & {\bf \color{blue}1} & 1 & 1  \\
\end{pNiceArray},  \;
\begin{pNiceArray}[small]{rrrrrrrrr}
\CodeBefore
 \cellcolor[HTML, opacity = 0.7]{FFFF88}{1-1, 1-2, 1-6, 1-7, 1-8, 2-1, 2-2, 2-6, 2-7, 2-8, 4-1, 4-2, 4-6, 4-7, 4-8}
\Body
{\bf \color{red}0} & {\bf \color{red}-1} & 0 & -1 & 0 & {\bf \color{blue}0} & {\bf \color{blue}-1} & 0 & 0 \\
{\bf \color{red}1} & {\bf \color{red}1}  & 0 & 0  & -1 & 0 & 0 & -1 & 0 \\
0 & 0  & 1 & 1  & 1 & 0 & 0 & 0 &  -1 \\
0 & 0   & 0  & 0   & 0 & {\bf \color{blue}1} & {\bf \color{blue}1} & 1 & 1  \\
\end{pNiceArray}.  
\end{equation*}
In addition, $ L_6 $ contains exactly $ \binom{6-3}{2} = 3 $ distinct submatrices equivalent to $ L_4 $, such that removing the first row causes the rank to drop. These submatrices are bold and colored in red, green, and blue. Their flats are $ \{0,1\} $, $ \{2,3\} $, $ \{5,6\} $.
\end{example}

\begin{proof}[Proof of Proposition \ref{prop: number_of_deg_flats}]
   %We begin by observing that the framed square matrices of $ L_m $, as in \eqref{matrix: Lm}, maintain their structure after deleting rows and columns indexed by the same~set.
      
  We want to construct a submatrix $K$ of $L_m $ equivalent to $ L_{m-r}$. Selecting the columns of $K$ means selecting all linear functions among 
  \begin{equation} \label{eq:linforms}  \{ x_i \, :\, 1 \leq i \leq n \} \cup \{ x_j-x_i \,:\, 0 \leq i < j \leq m-3 \},\end{equation}
   which involve only $m-r-2$ of the variables $x_0=1,x_1,\ldots, x_{m-3}$. Notice that each such submatrix $K$ is a flat of $M(L_m)$ of rank $m-r-2$. Indeed, the rank is that of $L_{m-r}$, and any column not contained in $K$ is a linear form among~\eqref{eq:linforms} which involves a new variable, so adding it to $K$ would increase the rank.  We claim that among these flats, the only flats of type (ii) are those containing the variable $x_0=1$. Recall that the flat $K$ is of type (ii) if deleting the first row decreases its rank. In terms of the linear forms \eqref{eq:linforms}, deleting the first row corresponds to setting $x_0 =0$. If $x_0$ is not among the variables in our flat, then this clearly does not change the rank, so the flat is of type (i). If $x_0$ is among the variables, then after setting $x_0 =0$ the rank is at most $m-r-3< m-r-2$.

Among the flats $K$ described above, precisely $ \binom{m-3}{m-r-3} = \binom{m-3}{r}$ many involve $x_0$. We have shown that these are the type {\rm (ii)} flats equivalent to $ L_{m-r} $.
\end{proof}

Below, we write $I_r(W), W\in \binom{[m-3]}{m-r-3}$ for the flats of $M(L_m)$ whose matrices are equivalent to $L_{m-r}$. More precisely, $W$ is an $(m-r-3)$-element subset of $\{1, \ldots,m-3\}$ and $I_r(W)$ is the rank-$(|W|+1)$ flat consisting of the linear forms 
\[ \{ x_i \,:\,i \in W \} \cup \{ x_j-x_i \, :\, j \in W, i \in \{0\} \cup W, i < j\}.\]

\begin{comment}

\begin{lemma} \label{lem: sublattice_flats}
The type (ii) flats $I_r(W)$  for $0 \leq r \leq m-4$ and $W \in \binom{[m-3]}{m-r-3}$ form a sublattice inside the lattice of flats of~$M(L_m)$. For every $r$, the number of chains in that sublattice connecting $I_0([m-3]) = \{ 0, 1, \ldots,n\}$ with $I_r(W)$ equals $r!$.
\end{lemma}
\begin{proof}
   We can easily verify that the intersection of two of our flats is again one of our flats: $I_r(W) \cap I_{r'}(W') = I_{m-3-|W\cap W'|}(W \cap W')$.
    In other words, these flats form a sublattice inside the lattice of flats of $ M(L_m) $.

    By Proposition \ref{prop: number_of_deg_flats}, there are exactly $ \binom{m-3}{r} $ distinct  flats $I_r(W)$ of type {\rm (ii)}. On the other hand, $ I_0([m-3])$ contains $ m-3 $ type {\rm (ii)} flats $ I_1(W)$, where $|W| = m-4$. Each of these, in turn, contains $ m-4 $ type {\rm (ii)} flats $ I_2(W)$, and so on. 
In total, we obtain $ (m-3) \cdots (m-2-r)$  flats  $I_r(W)$ with $|W| =m-3-r $, counted with repetitions. By comparing the numbers $ \binom{m-3}{r} $ and $ (m-3) \cdots (m-2-r) $, we find that each flat~$ I_r(W) $ is connected to $I_0([m-3])$ by $r!$ chains.
\end{proof}
\end{comment}

\begin{example} \label{ex:sublattice}
Since $I_r(W) \cap I_{r'}(W') = I_{m-3-|W\cap W'|}(W \cap W')$, the type (ii) flats $I_r(W)$ form a sublattice inside the lattice of flats of $M(L_m)$. This is illustrated for $M(L_6)$ in Figure~\ref{fig:sublattice_of_flats}. The cover relations can also be inferred from Example~\ref{ex: submatricesLm}. For instance, the submatrix $ L_4 $ colored in red corresponds to the flat~$ \{0,1\} $. It appears as a submatrix of the $L_5$ associated with the flat $ \{0,1,2,3,4\} $, and it also appears as a submatrix of the $ L_5 $ corresponding to the flat $ \{0,1,5,6,7\} $.
\end{example}
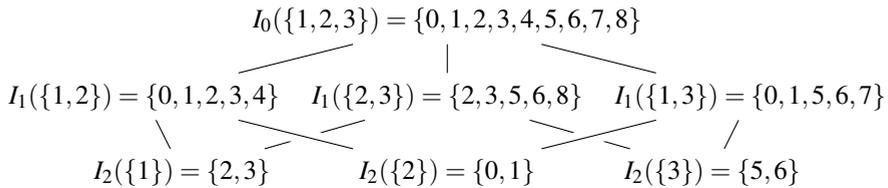
\begin{figure}
    \centering
    \begin{tikzpicture}[scale = 0.5]
  \node (max) at (0,4) {\small$I_0(\{1,2,3\})=\{0,1,2,3,4,5,6,7,8\}$};
  \node (a) at (-8,2) {\small$ I_1(\{1,2\}) = \{0,1,2,3,4\}$};
  \node (b) at (0,2) {\small$ I_1(\{2,3\}) =\{2,3,5,6,8\}$};
  \node (c) at (8,2) {\small$I_1(\{1,3\}) =\{0,1,5,6,7\}$};
  \node (d) at (-7,0) {\small$I_2(\{1\}) =\{2,3\}$};
  \node (e) at (0,0) {\small$I_2(\{2\}) =\{0,1\}$};
  \node (f) at (7,0) {\small$I_2(\{3\}) =\{5,6\}$};
  \draw  (d) -- (a) -- (max) -- (b) -- (f)
  (e)  (f) -- (c) -- (max)
  (d) -- (b);
  \draw[preaction={draw=white, -,line width=6pt}] (a) -- (e) -- (c);
\end{tikzpicture}
    \caption{Hasse diagram of type {\rm (ii)} flats $I_r(W)$ of $M(L_6)$.}
    \label{fig:sublattice_of_flats}
\end{figure}
By Theorem \ref{thm:flats} and the fact that $X_{I_r(W)} \simeq {\cal M}_{0,m-3-r}$, each flat $ I_r(W)$ contributes $(m-3-r)!$ points to the intersection $\mathbb{L}_s \cap {\cal R}_{L_m}$. These points lie in the open stratum ${\cal R}_{(L_m)_{I_r(W)}}^\circ$ corresponding to the flat $ I_r(W)$. They are the solutions to the equations ${A_{I_r(W)}^{\top} \mathrm{diag}(s) y_{I_r(W)} = 0}$, $ y \in {\cal R}_{(L_m)_{I_r(W)}}^\circ $. By the definition of the flats~$ I_{r}(W) $, this system is equivalent to the equations \eqref{eq:scatteringeqs} of the arrangement of $L_{m-r}$. These are the scattering equations of ${\cal M}_{0,m-r}$, with particles indexed by $1,2,m$ and $W$. We predict the intersection multiplicity of the solutions.
\begin{conjecture} \label{conj:multiplicity}
    The multiplicity of $\mathbb{L}_s \cap {\cal R}_{L_m}$ at each of the ${(m-3-r)!}$ points in $\mathbb{L}_s \cap {\cal R}_{(L_m)_{I_r(W)}}^\circ$ equals $r!$. 
\end{conjecture}

Conjecture \ref{conj:multiplicity} is supported by computations for small $m$, and we believe that the sublattice of flats in Example \ref{ex:sublattice} may be useful for proving it. Assuming our conjecture, we give a full description of the intersection $\mathbb{L}_s \cap \mathcal{R}_{L_m}$. %We are using the same notation as in Section 3.  %Let $L_m$ be as in \eqref{matrix: Lm}, let $A_m$ be the submatrix consisting of the last $m-3$ rows and let $\mathbb{L}_s =\{ A_m \, {\rm diag}(s) \,y =0 \}$. Here ${\rm diag}(s)$ is a diagonal matrix with diagonal entry $s_{ij}$ in the $k$-th position if the $k$-th column of $L_m$ is the $ij$ minor of \eqref{point_of_M0m}.

\begin{theorem} \label{thm:mainsec5}\begin{comment}
    All boundary solutions of the scattering equations on $\mathcal{R}_{L_m}$ are solutions to the scattering equations on $\mathcal{M}_{0,m-r}$ for $r=1,\ldots,m-4$. More precisely, there are $\mathrm{MLdeg}(\mathcal{M}_{0,m}) = (m-3)!$ non-boundary solutions in the open stratum  $\mathbb{L}_{s} \cap \mathcal{R}_{L_m}^{\circ}$ and all boundary solutions lie in 
    \[\bigsqcup\limits_{\substack{I_{m-r} \text{ is a flat of type {\rm (ii)}} \\ { \text{ of } M(L_{m}) \text{ for } m-3>r > 0}}} \mathbb{L}_{s} \cap \mathcal{R}_{L_{I_{m-r}}}^{\circ},\]
    \end{comment}
   Assume that Conjecture \ref{conj:multiplicity} holds. For generic $s \in \mathbb{C}^{\frac{m(m-3)}{2}}$, the set $\mathbb{L}_s \cap {\cal R}_{L_m}$ is finite and it decomposes~as
    \begin{equation} \label{eq:decomp} \mathbb{L}_s \cap {\cal R}_{L_m} \, = \, \bigsqcup_{r = 0}^{m-4} \bigsqcup_{W \in \binom{[m-3]}{m-3-r}} \mathbb{L}_s \cap {\cal R}_{(L_m)_{I_r(W)}}^\circ. \end{equation}
    The component $(r,W)$ in this decomposition consists of $(m-3-r)!$ points with multiplicity~$r!$. The non-zero coordinates $(y_i)_{i \in I_r(W)}$ of these points are the solutions to the scattering equations for the particles indexed by $1, 2, m$ and $W$. 
\end{theorem}
\begin{proof}
Conjecture \ref{conj:multiplicity} would imply that the component $(r,W)$ in the righthand side of \eqref{eq:decomp} consists of $(m-3-r)!$ solutions with multiplicity $ r! $. Therefore, 
\[
{\rm deg}(\mathbb{L}_{s} \cap \mathcal{R}_{L_m}) \, \geq \,  \sum\limits_{r=0}^{m-4} \binom{m-3}{r} r! (m-3-r)! = (m-3)(m-3)!.
\]
On the other hand, the lefthand side cannot exceed $ \deg \mathcal{R}_{L_m} =  (m-3)(m-3)! $ (Proposition \ref{prop:recdegM0m}). Thus, we have found all points in $\mathbb{L}_s \cap {\cal R}_{L_m}$.
\end{proof}

Via Theorems \ref{thm:flats} and \ref{thm:mainsec5}, Conjecture \ref{conj:multiplicity} would imply that the flats $I_r(W)$ are the only type (ii) flats of $M(L_m)$. Conversely, if these are the only type (ii) flats, then that implies the set-theoretic decomposition \eqref{eq:decomp} via Theorem \ref{thm:flats}.

\begin{example} \label{ex:M06}
For $m =6$, the intersection $\mathbb{L}_s \cap \mathcal{R}_{L_6}$ contains six distinct solutions of multiplicity one in $\mathcal{R}_{L_6}^{\circ}$,  whose coordinates are all non-zero. These are exactly the solutions to the scattering equations on $\mathcal{M}_{0,6}$. In addition, there are two roots of multiplicity one in each stratum $I_1(W)$, and one root of multiplicity $2$ in each stratum $I_2(W)$. In total, this accounts for $\deg \, {\cal R}_{L_6} = (6-3) (6-3)! = 18$ solutions. 
One can verify these numbers using our package \texttt{ProudfootSpeyerHomotopy} with the optional input \texttt{return\_boundary = true} in the function \texttt{solve\_PS} \cite{mathrepo}.
\end{example}
\begin{remark}\label{rem:badstart}
Unlike in Example \ref{ex:intro}, we should really use a generic matrix $ A_0^{\top}$ in the start system of a Proudfoot-Speyer homotopy for computing the intersection $\mathbb{L}_s \cap {\cal R}_{L_m}$, as prescribed by Algorithm \ref{alg:PShomotopy}. Picking $\omega = (9, \ldots, 1)$ for $m = 6$, the system $ A^\top \mathrm{diag}(s) y = 0$ has 8 solutions of multiplicity 1 and 5 solutions of multiplicity 2 on $V(J)$. The computation is found at \cite{mathrepo}. Hence, the start solutions are not regular, and not suitable for a homotopy continuation algorithm. 
\end{remark}

\section{Hilbert regularity} \label{sec:6}
The Proudfoot-Speyer homotopy in Section \ref{sec:4} is a numerical continuation method for solving the scattering equations associated to any hyperplane arrangement. It works inherently over the complex numbers, and uses floating point arithmetic. This section offers a more algebraic view. Let $K$ be a field of characteristic $0$, e.g., $\mathbb{Q}$, $\mathbb{R}$, $\mathbb{C}$ or  $\mathbb{Q}(u_0, \ldots, u_n)$. We assume that $L$ has entries in $K$ and study the \emph{Hilbert regularity} of the algebra $K[{\cal R}_L]/I_u$. We demonstrate through an example how this determines the size of \emph{Macaulay matrices} used for solving our equations via Gr\"obner basis and resultant methods. We start with definitions.

%This section investigates the complexity of solving the scattering equations (\ref{eq:scatteringeqs}) using symbolic methods. Gröbner basis computations are at the core of the most used techniques to solve a polynomial system, hence the complexity of computing the solutions is expressed as the complexity of the Gröbner basis computations. This is done using the Macaulay matrices \textcolor{red}{[ref]} up to a certain degree, known as the \emph{solving degree} of the ideal generated by the polynomial equations, see \cite{salizzoni2023upperboundsolvingdegree} for recent developments. However, computing the solving degree is often as hard as computing a Gröbner basis, so it is studied through other related invariants, such as the Hilbert regularity of the ideal. The main result of this section is Theorem \ref{thm:hilbert_intro}, that gives an upper bound for the regularity ${\rm Reg}(I(\mathbb{L}))$. Computational evidences show the bound is exact in all the tested cases. We recall the main definitions.

Let $R$ be a finitely generated $\mathbb{Z}$-graded $K$-algebra: $R = \bigoplus_{q\in \mathbb{Z}} R_q$. The reader should think of $R$ as the homogeneous coordinate ring $K[V] = K[x_0,\ldots,x_n]/I(V)$ of a projective variety $V \subset \mathbb{P}^n (\overline{K})$. The \emph{Hilbert function} of $R$ is
\[ {\rm HF}_R: \mathbb{Z}\longrightarrow \mathbb{Z}, \quad  {\rm HF}_R(q)= {\rm dim}_K (R_q) .\] 
A theorem by Hilbert \cite[Theorem~4.1.3]{Bruns_Herzog_1998} says that this function agrees with a polynomial for $q\gg 0$. This is called the \emph{Hilbert polynomial} of $R$, denoted by ${\rm HP}_R$. If $R =K[V]$ for an equidimensional projective variety $V \subset \mathbb{P}^n (\overline{K})$ of dimension $d$ and degree $k$, then ${\rm HP}_R(q)$ is a degree $d$ polynomial in $q$ with leading term $\frac{k}{d!}\,q^d$. The \emph{Hilbert regularity} of $R$ is the smallest degree from which the Hilbert function and the Hilbert polynomial agree:
    \[ {\rm HReg}(R) = \min\{ i \in \mathbb{Z} \,:\, {\rm HF} _{R}(q) = {\rm HP}_{R}(q) \ \text{ for every } \ q\geq i \}. \]
%We recall that the quotient of a graded algebra by a homogeneous ideal inherits the graded structure, so we can consider its Hilbert function.
%\begin{definition}
   % The \emph{Hilbert regularity} of a homogeneous ideal $I$ in a finitely generated $\mathbb{Z}-$graded  $\mathbb{K}-$algebra $R$ is the smallest integer from which the Hilbert function and the Hilbert polynomial agree:
   % \[ {\rm HReg}(I) = \min\{ i \in \mathbb{Z}: {\rm HF} _{R/I}(j) = {\rm HP}_{R/I}(j) \ \text{ for every } \ j\geq i \}. \]
%\end{definition}
%\noindent
%By Theorem \ref{thm:proudfoot-speyer}, the ideal $I(\mathcal{R}_L)$ is homogeneous and the broken circuit ring $\mathbb{C}[\mathcal{R}_L]$ is a graded algebra.
%We consider a linear subspace $\mathbb{L}\subseteq \mathbb{P}^n$ defined by $d$ linear equations in $\mathbb{C}[\mathcal{R}_L]$. Hence $\mathbb{L}\cap\mathcal{R}_L$ consists of ${\rm deg}(\mathcal{R}_L)$ many points and it is a $0-$dimensional variety. Its coordinate ring $\mathbb{C}[\mathcal{R}_L]/I(\mathbb{L})$ has Krull dimension equal to $1$ and the Hilbert polynomial is the constant ${\rm HP}_{ \mathbb{C}[\mathcal{R}_L]/I(\mathbb{L})}(i) = {\rm deg}(\mathcal{R}_L)$. 
%An upper bound of ${\rm HReg}(I(\mathbb{L}))$ is given by the Hilbert regularity of the ambient ring $\mathbb{C}[\mathcal{R}_L]$. 
All definitions above apply to the ring $K[{\cal R}_L]= K[x_0,\ldots,x_n]/I({\cal R}_L)$, where $I({\cal R}_L)$ is generated by the polynomials $f_C$ in Theorem \ref{thm:proudfoot-speyer} with coefficients in~$K$.

\begin{proposition}\label{prop: HReg broken circuit ring}
    The Hilbert regularity of ${\cal R}_L$ satisfies ${\rm HReg}(K[\mathcal{R}_L]) \leq 0$ and equality holds if and only if the matroid $M(L)$ is connected.
\end{proposition}
\begin{proof}
    The Hilbert regularity is read from the Hilbert series 
    \[ {\rm HS}_{K[\mathcal{R}_L]}(q) = \dfrac{h_{{K[\mathcal{R}_L]}}(q)}{(1-q)^{d+1}} \]
    as ${\rm HReg}(K[\mathcal{R}_L]) = {\rm deg}(h_{K[\mathcal{R}_L]}) - (d+1) +1$, see \cite[Proposition~4.1.12]{Bruns_Herzog_1998}. To compute the degree of the numerator we observe that the Hilbert series is left unchanged by a Gröbner degeneration (\cite[Theorem~1.6.2]{bruns2022determinants}). Thus, the Hilbert series of $K[{\cal R}_L]$ is that of the Stanley-Reisner ring $K[y_0,\ldots,y_n]/V(J)$. The~numerator has degree at most $d$ by \cite[Proposition 7.4.7(ii)]{bjorner1992homology} and~\cite[Section 2]{Proudfoot2006}. The degree is equal to $d$ if and only if the beta invariant $\beta(M(L))$ of the matroid is nonzero by \cite[Proposition 7.4.7(iii)]{bjorner1992homology}. The latter condition is equivalent to the matroid $M(L)$ being connected \cite[Proposition 7.4.8]{bjorner1992homology}.   %e Stanley-Reisner ring of Since $\mathbb{C}[\mathcal{R}_L]$ flatly degenerates to the Stanley-Reisner ring ${\rm SR}({\rm bc}_\prec(L))$, its Hilbert series is well-studied and the numerator has degree  equal the degree of the denominator, that is the dimension of the ring \cite[Cor.~1.15]{miller2004combinatorial}. Therefore we have ${\rm HReg}(\mathbb{C}[\mathcal{R}_L])=1. $
\end{proof}

Let $\mathbb{L} \subset \mathbb{P}^n({\overline{K}})$ be a linear space of dimension $n-d$, defined over $K$, so that $\overline{K}[{\cal R}_L]/I(\mathbb{L})$ has Krull dimension $1$. Let $h \in K[{\cal R}_L]_k$ be of degree $k$ and such that $\overline{K}[{\cal R}_L]/(I(\mathbb{L}) +\langle h \rangle)$ has Krull dimension $0$. Geometrically, this means that ${\mathbb{L} \cap {\cal R}_L}$ consists of finitely many points, and $h$ does not vanish at any of these. To emphasize this geometric interpretation we write $I(\mathbb{L} \cap V_h) = I(\mathbb{L}) +\langle h\rangle$.
\begin{theorem} \label{thm:regularitysec6}
Let $\mathbb{L}$ and $h$ be as above. We have 
    \begin{enumerate}
        \item ${\rm HReg}(K[{\cal R}_L]/I(\mathbb{L})) \leq d$ and ${\rm HF}_{K[{\cal R}_L]/I(\mathbb{L})}(q) = \deg {\cal R}_L$ for $q \geq d$,
        \item ${\rm HReg}(K[{\cal R}_L]/I(\mathbb{L} \cap V_h) \leq d+k$ and ${\rm HF}_{K[{\cal R}_L]/I(\mathbb{L} \cap V_h)}(q) = 0$ for $q \geq d +k$.
    \end{enumerate}
\end{theorem}

\begin{proof}
    Since $\mathcal{R}_L$ is arithmetically Cohen-Macaulay \cite{Proudfoot2006}, the first statement is a direct consequence of Proposition \ref{prop: HReg broken circuit ring} and \cite[Theorem 5.4]{betti2025solving}. The ideal $I(\mathbb{L})$ is generated by $d$ linear forms $f_1, \ldots, f_d \in {\cal R}_L$. The proof of \cite[Theorem 5.4]{betti2025solving} is easily adapted to the regular sequence $f_1, \ldots, f_d, h$ to show (ii). 
\end{proof}

\begin{comment}
\begin{proof}[Proof of Theorem \ref{thm:hilbert_intro}]
The \emph{saturation} of $I(\mathbb{L})$ is, by definition, the ideal  $I(\mathbb{L})^{\rm{sat}} = \{ f \in \mathbb{C}[\mathcal{R}_L]  : \mathbb{C}[\mathcal{R}_L]_i\cdot f \subset I(\mathbb{L}), \text{ for some } i\in \mathbb{Z} \}$. Since the variety $\mathcal{R}_L$ is arithmetically Cohen-Macaulay \cite{Proudfoot2006} and $I(\mathbb{L})$ is generated by ${\rm dim}(\mathcal{R}_L)$ many linear forms, we have that $I(\mathbb{L})$ is saturated, i.e. $I(\mathbb{L})= I(\mathbb{L})^{\rm {sat}}$ by Lemma 5.2 in \cite{betti2025solving}. Applying  \cite[Thm.~5.4]{betti2025solving} combined with Proposition \ref{prop: HReg broken circuit ring} we obtain the desired upper bound for the Hilbert regularity of $I(\mathbb{L})$: 

\[  {\rm HReg}(K[{\cal R}_L/I(\mathbb{L})) \leq d+ {\rm HReg}(K[\mathcal{R}_L])  \leq d. \]
By definition of Hilbert regularity this yields:
\[  {\rm HF}_{K[\mathcal{R}_L]/I(\mathbb{L})}(i) = {\rm HP}_{ K[\mathcal{R}_L]/I(\mathbb{L})}(i) = {\rm deg}(\mathcal{R}_L) \text{ for every } i\geq d. \]
\end{proof}
\noindent
Using \texttt{Macaulay2} we can see that the Hilbert function of $I(\mathbb{L})$ is not the degree of $\mathcal{R}_L$ for $i<d$ for any $2\leq d<n\leq 10$, so in these cases ${\rm HReg}(I(\mathbb{L})) = d$.
\end{comment}

Theorem \ref{thm:regularitysec6} implies Theorem \ref{thm:hilbert_intro}. A detailed investigation of the implications of Theorem \ref{thm:regularitysec6} for symbolic solutions to the scattering equations is beyond the scope of this paper. We illustrate its use by means of an example, in which we construct a \emph{Macaulay matrix} to study the intersection $\mathbb{L}_u \cap {\cal R}_L \subset \mathbb{P}^n(\overline{K})$. 

\begin{example}
    We turn back to Example \ref{ex:intro}. Let $K =\mathbb{Q}(u_0,u_1,u_2,u_3,t)$ be the field of rational functions in the exponents $u$ and a new variable $t$. Let $h =y_2-ty_1$. The determinant of the following $19 \times 19$ matrix 
    % MacaulayMatrix in degree 3 for the scattering equations and y2-t*y1 on the broken circuit ring of the arrangement defined by uniform matroid U_3,4.
 \setlength\arraycolsep{1.5pt}
\[M = \tiny \begin{array}{l}\vphantom{u_{0}0-u_{2}-2\,u_{3}000000000000000}\\\vphantom{0u_{0}000-u_{2}-2\,u_{3}000000000000}\\\vphantom{00u_{0}00-2\,u_{3}2\,u_{3}-u_{2}00000-2\,u_{3}00000}\\\vphantom{000u_{0}0-u_{2}u_{2}0-2\,u_{3}0000-u_{2}00000}\\\vphantom{0000u_{0}00000-u_{2}-2\,u_{3}0000000}\\\vphantom{00000u_{0}000000-u_{2}-2\,u_{3}00000}\\\vphantom{000000u_{0}000000-u_{2}-2\,u_{3}0000}\\\vphantom{0000000u_{0}0000000-u_{2}-2\,u_{3}00}\\\vphantom{00000u_{0}-u_{0}000000u_{0}00-u_{2}-2\,u_{3}0}\\\vphantom{00000000u_{0}00000000-u_{2}-2\,u_{3}}\\\vphantom{0u_{1}-2\,u_{2}-u_{3}000000000000000}\\\vphantom{0000u_{1}-2\,u_{2}-u_{3}000000000000}\\\vphantom{00000u_{1}-u_{3}u_{3}-2\,u_{2}00000-u_{3}00000}\\\vphantom{000000000u_{1}-2\,u_{2}-u_{3}0000000}\\\vphantom{0000000000u_{1}0-2\,u_{2}-u_{3}00000}\\\vphantom{000000000000u_{1}00-2\,u_{2}-u_{3}00}\\\vphantom{0-t10000000000000000}\\\vphantom{0000-t10000000000000}\\\vphantom{000000000-t100000000}\end{array}\left(\!\begin{array}{ccccccccccccccccccc}
      \vphantom{}u_{0}&0&-u_{2}&-2\,u_{3}&0&0&0&0&0&0&0&0&0&0&0&0&0&0&0\\
      \vphantom{}0&u_{0}&0&0&0&-u_{2}&-2\,u_{3}&0&0&0&0&0&0&0&0&0&0&0&0\\
      \vphantom{}0&0&u_{0}&0&0&-2\,u_{3}&2\,u_{3}&-u_{2}&0&0&0&0&0&-2\,u_{3}&0&0&0&0&0\\
      \vphantom{}0&0&0&u_{0}&0&-u_{2}&u_{2}&0&-2\,u_{3}&0&0&0&0&-u_{2}&0&0&0&0&0\\
      \vphantom{}0&0&0&0&u_{0}&0&0&0&0&0&-u_{2}&-2\,u_{3}&0&0&0&0&0&0&0\\
      \vphantom{}0&0&0&0&0&u_{0}&0&0&0&0&0&0&-u_{2}&-2\,u_{3}&0&0&0&0&0\\
      \vphantom{}0&0&0&0&0&0&u_{0}&0&0&0&0&0&0&-u_{2}&-2\,u_{3}&0&0&0&0\\
      \vphantom{}0&0&0&0&0&0&0&u_{0}&0&0&0&0&0&0&0&-u_{2}&-2\,u_{3}&0&0\\
      \vphantom{}0&0&0&0&0&u_{0}&-u_{0}&0&0&0&0&0&0&u_{0}&0&0&-u_{2}&-2\,u_{3}&0\\
      \vphantom{}0&0&0&0&0&0&0&0&u_{0}&0&0&0&0&0&0&0&0&-u_{2}&-2\,u_{3}\\
      \vphantom{}0&u_{1}&-2\,u_{2}&-u_{3}&0&0&0&0&0&0&0&0&0&0&0&0&0&0&0\\
      \vphantom{}0&0&0&0&u_{1}&-2\,u_{2}&-u_{3}&0&0&0&0&0&0&0&0&0&0&0&0\\
      \vphantom{}0&0&0&0&0&u_{1}-u_{3}&u_{3}&-2\,u_{2}&0&0&0&0&0&-u_{3}&0&0&0&0&0\\
      \vphantom{}0&0&0&0&0&0&0&0&0&u_{1}&-2\,u_{2}&-u_{3}&0&0&0&0&0&0&0\\
      \vphantom{}0&0&0&0&0&0&0&0&0&0&u_{1}&0&-2\,u_{2}&-u_{3}&0&0&0&0&0\\
      \vphantom{}0&0&0&0&0&0&0&0&0&0&0&0&u_{1}&0&0&-2\,u_{2}&-u_{3}&0&0\\
      \vphantom{}0&-t&1&0&0&0&0&0&0&0&0&0&0&0&0&0&0&0&0\\
      \vphantom{}0&0&0&0&-t&1&0&0&0&0&0&0&0&0&0&0&0&0&0\\
      \vphantom{}0&0&0&0&0&0&0&0&0&-t&1&0&0&0&0&0&0&0&0
      \end{array}\!\right)\]
    \noindent
satisfies $(\det \, M)(u,t) = P(u) \cdot Q(u,t)$,
    where $Q(u,t)$ is an irreducible polynomial of degree $3$ in $t$. The roots of $Q$ are algebraic functions in $u$, which are the values of the rational function $\frac{y_2}{y_1}$ at the points in $\mathbb{L}_u \cap {\cal R}_L$, i.e., the points satisfying \eqref{eq:scatteringexample}. In particular, normalizing $Q(u,t)$ to a monic polynomial in $t$, we read the values of the elementary symmetric functions at the $\frac{y_2}{y_1}$ coordinates of the scattering solutions from its coefficients. We justify this claim via Theorem~\ref{thm:regularitysec6}.
    
    The ring $K[{\cal R}_L]$ is the quotient of $K[y_0,y_1,y_2,y_3]$ by the principal ideal of a cubic, seen in \eqref{eq:scatteringexample}. The Hilbert function of $K[{\cal R}_L]$ for $q = 0, 1, 2, 3$ is given by $1,4,10,19$. A basis for the 19-dimensional $K$-vector space $K[{\cal R}_L]_3$ is 
    \begin{equation} \label{eq:basisRL} \begin{matrix} y_{0}^{3},\:y_{0}^{2}y_{1
      },\:y_{0}^{2}y_{2},\:y_{0}^{2}y_{
      3},\:y_{0}y_{1}^{2},\:y_{0}y_{1}y
      _{2},\:y_{0}y_{1}y_{3},\:y_{0}y_{
      2}^{2},\:y_{0}y_{3}^{2}, \\ \:y_{1}^{
      3},\:y_{1}^{2}y_{2},\:y_{1}^{2}y
      _{3},\:y_{1}y_{2}^{2},\:y_{1}y_{2
      }y_{3},\:y_{1}y_{3}^{2},\:y_{2}^{
      3},\:y_{2}^{2}y_{3},\:y_{2}y_{3
      }^{2},\:y_{3}^{3}.\end{matrix}\end{equation}
    %The ideal $I({\mathbb{L}}) = \langle u_0y_0 -u_2y_2-2u_3y_3 , u_1y_1-2u_2y_2-u_3y_3 \rangle$ defines a quotient $K[{\cal R}_L]/I(\mathbb{L})$ of Krull dimension $1$ ($\mathbb{L} \cap {\cal R}_L$ is finite). 
    By Theorem \ref{thm:regularitysec6}(ii), we can find 19 generators of $(I(\mathbb{L}_u) +\langle h \rangle)_3$ such that their expansions in the basis \eqref{eq:basisRL} of $K[{\cal R}_L]$ give an invertible matrix over $K$. That matrix is $M$. The first 16 rows represent a basis of $I(\mathbb{L}_u)_3$, which has codimension three in $K[{\cal R}_L]_3$ by Theorem \ref{thm:regularitysec6}(i). If we specialize $t$ to the value of $\frac{y_2}{y_1}$ at a point $y$ in $\mathbb{L}_u \cap {\cal R}_L$, then $h$ vanishes at $y$. Evaluating the basis monomials \eqref{eq:basisRL} at~$y$ gives a non-zero kernel vector of $M$, which shows that $(\det \,  M)(u,t) =0$.

    We note that replacing $h$ by $h_2(y)-th_1(y)$ for any non-zero linear forms $h_1, h_2 \in K[{\cal R}_L]_1$ only changes the last three rows of $M$, and the roots of its determinant are the values of $h_2/h_1$ at the three solutions. Increasing the degree of $h$ to $k$ would increase the size of the matrix to ${\rm HF}_{{\cal R}_L}(d+k)$, and allows to evaluate more complicated rational functions and their traces. For instance, one can evaluate the CHY amplitude by choosing $h_1$ and $h_2$ to be the numerator and denominator of the toric Hessian determinant of ${\cal L}_u$, as in \cite[Theorem 13]{sturmfels2021likelihood}. This computation is implemented in the \texttt{CHYamplitude.m2} file available at \cite{mathrepo}. 
\end{example}

\paragraph{Acknowledgements.}
We thank Emanuele Delucchi and Cynthia Vinzant for helpful conversations and for useful pointers to the literature. \,

\vspace{0.5cm}
\begin{footnotesize}
\noindent {\bf Funding statement}:
This project started at a workshop held at MPI MiS Leipzig, supported by the European Union (ERC, UNIVERSE PLUS, 101118787).
Views and opinions expressed are however those of the authors only and do not
necessarily reflect those of the European Union or the European Research Council
Executive Agency. Neither the European Union nor the granting authority can
be held responsible for them.
\end{footnotesize}

\vspace{-1cm}

\bibliography{references}

\end{document}